\documentclass{amsart}

\usepackage[T1]{fontenc}

\usepackage{graphicx}

\usepackage{hyperref}
\usepackage[latin1]{inputenc}
\usepackage{amssymb}
\usepackage{amsmath}
\usepackage{esint}
\usepackage{leftidx}

\newif\ifdraft
\draftfalse
\ifdraft
 \usepackage[notcite,notref]{showkeys} 
\usepackage[usenames,dvipsnames]{color}
\fi
\newcommand{\jp}[1]{{\ifdraft \color{blue} \fi #1}}

\newcommand{\rn}{{\mathbb{R}^n}}

\newcommand{\pp}{\partial}
\newcommand{\rl}{\partial_t^\alpha}
\newcommand{\ppa}{\leftidx{^C}\partial_t^\alpha}
\newcommand{\w}{\omega}
\newcommand{\W}{\Omega}
\newcommand{\eps}{\varepsilon}

\newcommand{\dtd}{\overline{\partial_{\tau}}}

\newcommand\restr[2]{{
  \left.\kern-\nulldelimiterspace 
#1 
  \vphantom{\big|} 
  \right|_{#2} 
  }}

\newcommand{\subf}[2]{%
  {\small\begin{tabular}[t]{@{}c@{}}
  #1\\#2
  \end{tabular}}%
}

\def\R{{\mathbb {R}}}
\def\N{{\mathbb {N}}}

\newtheorem{theorem}{Theorem}[section]
\newtheorem{lemma}[theorem]{Lemma}
\newtheorem{proposition}[theorem]{Proposition}

\theoremstyle{remark}
\newtheorem{remark}[theorem]{Remark}

\theoremstyle{definition}

\numberwithin{equation}{section}

\title[FE approximations for fractional evolution problems]{Finite Element approximations for fractional evolution problems}

\author[G. Acosta, F.M. Bersetche and J.P. Borthagaray]{Gabriel Acosta, Francisco M. Bersetche and Juan Pablo Borthagaray}

\address[G. Acosta, F.M. Bersetche and J.P. Borthagaray]{IMAS - CONICET and Departamento de Matem\'a\-tica, FCEyN - Universidad de Buenos Aires, Ciudad Universitaria, Pabell\'on I  (1428) Buenos Aires, Argentina.}

\email[G. Acosta]{gacosta@dm.uba.ar}

\urladdr[G. Acosta]{http://mate.dm.uba.ar/~gacosta/}

\email[F. M. Bersetche]{fbersetche@dm.uba.ar}

\email[J.P. Borthagaray]{jpbortha@dm.uba.ar}


\begin{document}
\begin{abstract}
This work introduces and analyzes a finite element scheme for evolution problems involving fractional-in-time and in-space 
differentiation operators up to order two. 
The left-sided fractional-order derivative in time we consider is employed to represent memory effects, while a nonlocal 
differentiation operator in space accounts for long-range dispersion processes.
We discuss well-posedness and obtain regularity estimates for the evolution problems under consideration. 
The discrete scheme we develop is based on piecewise linear elements for the space variable and a convolution quadrature for the time component. We illustrate the method's performance with numerical experiments in one- and two-dimensional domains.
\end{abstract}

\maketitle

\section{Introduction}

Since the introduction of Continuous Time Random Walks (CTRW) by Montroll and Weiss \cite{MontrollWeiss}, anomalous diffusion phenomena has been an active area of research among the scientific community. 
The CTRW assign a joint space-time distribution to individual particle motions: when the tails of these distributions are heavy enough, non-Fickian dispersion results for all time and space scales. A heavy-tailed jump (waiting time) distribution implies the absence of a characteristic space (time) scale.

The equivalence between these heavy-tailed motions and transport equations that use fractional-order derivatives has been shown by several authors; see, for example \cite{Hanert}.  Space nonlocality is a direct consequence of the existence of arbitrarily large jumps in space, whereas time nonlocality is due to the history dependence introduced in the dynamics by the presence of anomalously large waiting times.

The evidence of anomalous diffusion phenomena has been thoroughly reported in physical and social environments, such as plasma turbulence \cite{delCastillo, delCastillo1}, hidrology \cite{BensonWheatcraft, Berkowitz, Pachepsky}, finance \cite{Mainardi}, and human travel \cite{Brockmann} and predator search \cite{Sims} patterns.
Models of transport dynamics in complex systems taking into account this non-Fickian behavior have been proposed accordingly. Also, evolution processes intermediate between diffusion and wave propagation have been shown to govern the propagation of stress waves in viscoelastic materials \cite{Freed, MainardiParadisi}.

Integer-order differentiation operators are local, because the derivative of a function at a given point depends only on the values of the function in a neighborhood of it. In contrast, fractional-order derivatives are nonlocal, integro-differential operators. A left-sided fractional-order derivative in time may be employed to represent memory effects, while a nonlocal differentiaton operator in space accounts for long-range dispersion processes.

We now describe the problems we are going to consider in this work.  Let $\W \subset \rn$ be a domain with smooth enough boundary,  $\alpha \in (0,2]$, $s \in (0,1)$ and a forcing term $f\colon \W \times (0,T) \to \R$. We aim to solve the fractional differential equation 
\begin{equation}
\label{eq:parabolic}
      \ppa u +  (-\Delta)^s u = f  \mbox{ in }\W\times(0,T).
\end{equation}
Here,  $\ppa$ denotes the Caputo derivative, given by
\begin{equation*} \label{eq:caputo}
\ppa u (x,t) =
\left\lbrace
  \begin{array}{rl}
	 \frac{1}{\Gamma(k-\alpha)} \int_0^t \frac{1}{(t-r)^{\alpha-k+1}} \frac{\pp^k u}{\pp r^k} (x,r) \, dr & \mbox{ if } k-1 < \alpha < k, \ k \in \N, \\
	\frac{\pp^k u}{\pp t^k} u(x,t) &  \mbox{ if } \alpha = k \in \N,
      \end{array}
    \right.
    \end{equation*}
while $(-\Delta)^s$ is the fractional Laplace operator, defined as
\begin{equation}
(-\Delta)^s u (x) = C(n,s) \mbox{ P.V.} \int_\rn \frac{u(x)-u(y)}{|x-y|^{n+2s}} \, dy.
\label{eq:fraccionario}
\end{equation}
Above, $ C(n,s) = \frac{2^{2s} s \Gamma(s+\frac{n}{2})}{\pi^{n/2} \Gamma(1-s)} $ is a normalization constant.

Closely related to the Caputo derivative, the Riemann-Liouville fractional derivative is needed in the sequel.
Let us recall here its definition,
 \begin{equation*} \label{eq:RL}
\rl u (x,t) =
\left\lbrace
  \begin{array}{rl}
	 \frac{1}{\Gamma(k-\alpha)}  \frac{\pp^k }{\pp t^k} \int_0^t \frac{1}{(t-r)^{\alpha-k+1}} u (x,r) \, dr & \mbox{ if } k-1 < \alpha < k, \ k \in \N, \\
	\frac{\pp^k u}{\pp t^k} u(x,t) &  \mbox{ if } \alpha = k \in \N.
      \end{array}
    \right.
    \end{equation*}
   
For $0<\alpha\le 1$, problem \eqref{eq:parabolic} is usually called a \emph{fractional diffusion} equation. On the
other hand, for  $1<\alpha\le2$  it is sometimes called a \emph{fractional diffusion-wave} equation.
Analyzing scaling and similarity properties of the Green function $G_{\alpha,s}$ associated to the 
operator $\ppa + (-\Delta)^s$, in \cite{Mainardi_fundamental} it is shown that 
$$G_{\alpha,s}(x,t) = t^{\frac{-\alpha}{2s}} \Phi_{\alpha,s} \left( \frac{x}{t^{\frac{\alpha}{2s}}}\right), $$
for a certain one-variable function $\Phi_{\alpha,s}$.
Notice that in case $\alpha = s$, although the CTRW associated to equation \eqref{eq:parabolic} has 
the same scaling properties as Brownian motion, the lack of finite moments  makes the diffusion process anomalous.
On the other hand, the term \emph{fractional wave} equation has been utilized to refer to the problem with $1<\alpha=2s<2$,  since for this choice of the parameters some features of the standard wave equation are preserved \cite{Luchko}. \jp{For example, the maximum, gravity center and mass center of the fundamental solution such possess constant propagation velocity.}

\jp{Let $v, b \in L^2(\W)$ be given data. We complement problem \eqref{eq:parabolic} with} the initial and boundary value conditions
\begin{equation} \label{eq:boundary}
 \left\lbrace
  \begin{array}{rl}
  	u =  0 & \mbox{ in }\W^c \times [0,T) , \\
	u(\cdot, 0) = v & \mbox{ in }\W ,
      \end{array}
    \right.
\end{equation}
and, additionally for $1<\alpha\le 2$, we require that 
\begin{equation} \label{eq:boundary_wave}
      \pp_t u (\cdot, 0)=  b  \mbox{ in }\W .
\end{equation}

A variational formulation for the time-fractional problem involving a Caputo derivative of order $\alpha \in (0,1)$ has been recently studied by Karkulik \cite{Karkulik}. In that work, the author shows  that the Caputo derivative is a linear and bounded operator on a time-fractional Sobolev-Bochner space, considers a variational formulation  based exclusively on Sobolev regularity and proves that if the initial condition belongs to $H^{1-1/\alpha + \eps}(\W)$ for some $\eps > 0$, then the time-fractional problem is well posed.

As fractional-order differential operators involve singular kernels, numerical approximations of equations involving them is a delicate task. Moreover, their nonlocal character calls for the design of efficient numerical schemes for the discretization.  Several strategies have been shown to succeed in the one-dimensional context. For example, finite element schemes were treated in \cite{FixRoop}; finite difference methods have been employed to discretize either the fractional-order space diffusion term \cite{Meerschaert, Meerschaert1} and time derivative \cite{LinXu}; pseudospectral methods were considered in \cite{Hanert}.
A comparison between these approaches may be found in \cite{HanertPiret}. Also, a convolution quadrature rule has been proposed in \cite{Lub2,Lub1}, and in turn, has been utilized for the approximation of fractional time derivatives in  \cite{Jin}.

Nevertheless, analysis and implementation of two-dimensional schemes is scarce. In \cite{AcostaBorthagaray}, the Dirichlet problem for the fractional Laplace operator was analyzed and a simple finite element implementation for the two-dimensional problem was introduced. Later on, the authors presented the code employed for such implementation in \cite{ABB}. The space discretization in the current work is based on such a code. \jp{Recently, implementations based on finite elements \cite{AG2, AG} and integral representation formulas \cite{BLP} have been proposed as well.}

It is noteworthy that the fractional Laplace operator defined by \eqref{eq:fraccionario} does not coincide with the operator considered, for example, in \cite{BonitoLeiPasciak, Nochetto_parabolico, Yang}. Indeed, the spatial operator considered in those works is a power of the Laplacian in the spectral sense. 

Our work does not include the case $s=1$, which corresponds to a local-in-space process, as it is already covered by other authors' work. For the range $0<\alpha\le 1$, reference \cite{JinLazarovZhou} develops a semidiscrete Galerkin method and studies the error both for smooth and non-smooth data. Naturally, the local-in-space case is also covered by the previously mentioned works \cite{BonitoLeiPasciak, Nochetto_parabolico, Yang} regarding spectral fractional powers of the Laplacian.  For the full range of time derivatives we are considering in this work, \cite{McLeanThomee} deals with an alternative formulation of \eqref{eq:parabolic} and a method based on the Laplace transform is developed, while in \cite{MustaphaMcLean} an approach via discontinuous Galerkin discretization in time is introduced.

\subsection*{Organization of the paper} Preliminary concepts regarding fractional Sobolev spaces, elliptic regularity for the fractional Laplacian and the Mittag-Leffler function are discussed in Section \ref{sec:preliminaries}. Moreover, that section also deals with the well-posedness of \eqref{eq:parabolic} with conditions \eqref{eq:boundary}, \eqref{eq:boundary_wave}. Afterwards, in Section \ref{sec:regularity} we take advantage of the representation of solutions to derive regularity estimates for the fractional evolution problems under consideration. A numerical scheme, based on standard Galerkin finite element approximations for the space variable and a convolution quadrature for the time component, is proposed 
in Section \ref{sec:discretization}, and
error bounds for this scheme are presented in Section \ref{sec:error}. In Section \ref{sec:numerical} we present some numerical examples that illustrate the accuracy of our convergence estimates \jp{and qualitative behavior of solutions to fractional evolution problems. Finally, Appendices \ref{sec:CQ} and \ref{sec:Err} include details on the convolution quadrature rule utilized for the time discretization and on the error analysis of the semi-discrete scheme, respectively.
}

\section{Preliminaries} \label{sec:preliminaries}
In this section we set the basic notation and present some preliminary concepts necessary for the analysis of the fractional evolution problems under consideration. We recall elliptic regularity results for the fractional Laplacian and some important properties of the Mittag-Leffler function. These properties are then utilized to derive a representation formula for solutions that allows to prove the well-posedness of problem  \eqref{eq:parabolic}.

\subsection{Sobolev spaces and fractional Laplace operator}
Let $\W \subset \rn$ be an open set and $s \in(0,1)$. The fractional Sobolev space $H^s(\W)$ is defined by
\[
H^s(\W) = \left\{ w \in L^2(\W) \colon |w|_{H^s(\W)} := \left( \iint_{\W^2}  \frac{|w(x)-w(y)|^2}{|x-y|^{n+2s}} \, dx \, dy \right)^{\frac12} < \infty \right\}.
\]
This set, furnished with the norm $\|\cdot\|_{H^s(\W)} = \|\cdot\|_{L^2(\W)} + |\cdot|_{H^s(\W)} ,$ constitutes a Hilbert space. 
If $s>1$ and it is not an integer, considering the decomposition $s = m + \sigma$, where $m \in \mathbb{N}$ and $\sigma \in (0,1)$, 
allows to define $H^s(\W)$ by means of
\[
H^s(\W) = \left\{ w \in H^m(\W) \colon |D^\alpha w|_{H^{\sigma}(\W)} < \infty \text{ for all } \alpha \text{ s.t }  |\alpha| = m \right\}.
\]
 
Another important space of interest for the problems under consideration is that of functions in $H^s(\rn)$ supported within $\overline \W$, 
\[
\widetilde{H}^s (\W) = \left\{ w \in H^s(\rn) \colon \text{ supp } w \subset \bar{\W} \right\}.
\]	
The bilinear form 
\begin{equation} \label{eq:pint}
\langle u , w \rangle_s := C(n,s) \iint_{(\rn \times \rn) \setminus (\W^c \times \W^c)} \frac{(u(x)-u(y)) (w(x)-w(y))}{|x-y|^{n+2s}} \, dy dx 
\end{equation}
constitutes an inner product on $\widetilde{H}^s(\W)$. The norm it induces, which is just a multiple of the $H^s(\rn)$-seminorm, is equivalent to the full $H^s(\rn)$-norm on this space, because a Poincar\'e-type inequality holds in it. See, for example, \cite{AcostaBorthagaray} for details.

\begin{remark}
We emphasize that functions in $\widetilde{H}^s(\W)$ are defined in the whole space $\rn$. In particular, this allows to consider the action of the fractional Laplacian over this set.
\end{remark}

Finally, Sobolev spaces of negative order are defined by duality, using $L^2(\W)$ as pivot space. Of interest in the problems we are considering is the space
$$
H^{-s}(\W) = \left( \widetilde{H}^s (\W) \right) '.
$$

Recalling definition \eqref{eq:pint} and denoting by $\left( \cdot, \cdot \right)$ the inner product in $L^2(\W)$, the weak formulation of problem \eqref{eq:parabolic} reads: find $u \in L^2(0,T; \widetilde{H}^s(\W))$ such that 
$$ \int_0^T \left( \ppa u (\cdot, t) \, w(\cdot, t) \right) dt +  
\int_0^T \langle u(\cdot, t), w(\cdot, t) \rangle_s \, dt = 
\int_0^T \left( f(\cdot, t), w(\cdot, t) \right) dt  
$$
for all $w \in L^2(0,T; \widetilde{H}^s(\W)).$

\subsection{Elliptic regularity}
We recall regularity results for the homogeneous problem
\begin{equation} \label{eq:homogeneous}
\left\lbrace
  \begin{array}{rl}      
(-\Delta)^s u  = g &   \text{in } \W, \\
u =  0 &  \text{in } \W^c.
\end{array}
    \right.
\end{equation}

Even though the fractional Laplacian is an operator of order $2s$ in $\rn$, in the sense that $(-\Delta)^s \colon H^\ell(\rn) \to H^{\ell - 2s}(\rn)$ is bounded and invertible, theory is much more delicate for problems posed on bounded domains.

Grubb \cite{Grubb} provides regularity estimates for solutions of \eqref{eq:homogeneous} in the setting of H\"ormander $\mu-$spaces. We express these results in terms of standard Sobolev spaces, and refer to \cite{Grubb} for further details.
\begin{proposition}\label{prop:regHr}
Let $\W \subset \rn$ be a bounded domain with smooth boundary, $g \in H^r(\Omega)$ for some $r\geq -s$ and consider $u\in \widetilde{H}^s(\W)$, the solution of the Dirichlet problem \eqref{eq:homogeneous}. Then, there exists a constant $C(n,s)$ such that 
$$
|u|_{H^{s+\gamma}(\rn)} \leq C \| g \|_{H^r(\W)}.
$$
Here, $\gamma = \min \{ s, 1/2 - \eps\}$, with $\eps > 0$ arbitrary small. 
\end{proposition}

\begin{remark} Observe that, in general, it is not true that solutions of \eqref{eq:homogeneous} have $2s$ more derivatives than the right-hand side function $g$.
No matter how regular $g$ is, the solution of \eqref{eq:homogeneous} is not expected to be more regular than $H^{s+\gamma}(\rn)$. 
In spite of this, the singular behavior of solutions can be localized at the boundary and described more appropriately
in weighted spaces \cite{AcostaBorthagaray,ABBM}.
\end{remark}

\begin{remark} \label{remark:mapping}
In view of Proposition \ref{prop:regHr}, given $v \in \widetilde{H}^s(\W)$, assuming that it satisfies $(-\Delta)^s v \in L^2(\W)$ is weaker than assuming that $v \in H^{2s}(\rn)$. Hypotheses of this type on the initial/boundary conditions are employed throughout this work.
\end{remark}

\subsection{Mittag-Leffler function} Let $\alpha > 0$ and $\mu \in \R$, then the Mittag-Leffler function  $E_{\alpha,\mu} \colon \mathbb{C} \to \mathbb{C}$ is defined by
\begin{equation*}
E_{\alpha,\mu}(z) = \sum_{k=0}^\infty \frac{z^k}{\Gamma(\alpha k + \mu)}.
\label{eq:mittag-leffler}
\end{equation*}

This is a complex function that depends on two parameters; in particular, it generalizes the exponentials, in view of the identity $E_{1,1}(z) = e^z$ for all $z \in \mathbb{C}$.
The following properties of this family of functions are useful to derive the regularity estimates we present in Section \ref{sec:regularity}.

\begin{lemma}[cf. {\cite[Theorem 4.3]{Diethelm}}] \label{lemma:der_mittag}
If $\alpha, \lambda > 0$, then 
\begin{equation*}
\ppa E_{\alpha, 1} (-\lambda t^\alpha ) = - \lambda E_{\alpha, 1} (-\lambda t^\alpha ) .
\label{eq:der_mittag}
\end{equation*}
Moreover, the following identities hold for integer-order differentiation:
\begin{equation*} 
\frac{d}{dt} E_{\alpha,1}(-\lambda t^\alpha) = -\lambda t^{\alpha-1} E_{\alpha,\alpha}(-\lambda t^\alpha) ,
\end{equation*}
and
\begin{equation*} \label{eq:der_mittag2}
\frac{d}{dt} \left(t E_{\alpha,2}(-\lambda t^\alpha) \right) = E_{\alpha,1}(-\lambda t^\alpha) .
\end{equation*}
\end{lemma}

\begin{lemma}[cf. {\cite[Theorem 1.6]{Podlubny}}] \label{lemma:bound_mittag}
 Let $0< \alpha<2$ and $\mu \in \R$ satisfying $\pi \alpha / 2 < \delta < \min \{ \pi, \pi\gamma\},$ then there exists a constant  $C(\alpha,\mu,\delta) > 0$ such that
\begin{equation*}
| E_{\alpha, \mu} (z)| \le \frac{C}{1 + |z|} \quad \mbox{for } \delta < | \arg(z)| \leq \pi.
\label{eq:est_mittag}
\end{equation*}
\end{lemma}

\subsection{Solution representation}
Let $\{(\phi_k, \lambda_k) \}_{k=1}^\infty$ denote the solutions of the fractional eigenvalue problem
\begin{equation} \label{eq:eigenvalues}
\left\lbrace
  \begin{array}{rl}
      (-\Delta)^s u = \lambda \, u & \mbox{in }\W, \\
      u =  0  \ \  \,& \mbox{in }\W^c . \\
      \end{array}
    \right.
\end{equation}
It is well-known that the fractional Laplacian has a sequence of eigenvalues
$$0 < \lambda_1 < \lambda_2 \le \ldots , \quad \lambda_k \to \infty \mbox{ as } k\to \infty ,
$$
and that the eigenfunctions' set $\{ \phi_k \}_{k = 1}^\infty$ may be taken to  constitute an orthonormal basis of $L^2(\W)$.  

\begin{remark} Unlike the classical Laplacian, eigenfunctions of the fractional Laplacian are in general non-smooth \cite{grubb_autovalores, RosOtonSerra2}. Indeed, considering a smooth function  $d$ that behaves like $\text{dist}(x,\partial \W)$ near to $\pp\W$, all eigenfunctions $\phi_k$ belong to the space $d^sC^{2s(-\eps)}(\W)$ (the $\eps$ is active only if $s=1/2$) and  $\frac{\phi_k}{d^s}$ does not vanish near $\partial\W$. 
The best Sobolev regularity guaranteed for solutions of \eqref{eq:eigenvalues} is $\phi_k \in H^{s + 1/2 - \eps}(\rn)$ for $\eps>0$ (see \cite{BdPM}).

The reduced Sobolev regularity of eigenfunctions precludes the possibility of solutions to equation \eqref{eq:parabolic} being smooth, even for $\alpha = 1$. This is in stark contrast with the case of the classical Laplacian. However, solutions of diffusion equations with memory --local in space but fractional in time-- are known to be less regular than their classical counterparts \cite{McLean}. The effect of fractional differentiation in time is that high-frequency modes are less strongly damped than in classical diffusion, and the time derivatives of the solution are unbounded as $t\to 0$. In the next section we shall show that solutions of \eqref{eq:parabolic} present the same behavior. \jp{See also the numerical experiments in Subsection \ref{ss:qualitative}.}
\end{remark}

We write solutions of \eqref{eq:parabolic} by means of separation of variables,
\begin{equation}
u(x,t) = \sum_{k = 1}^\infty u_k (t) \, \phi_k (x) .
\label{eq:sep_variables}
\end{equation}
Then, for every $k \ge 1$ it must hold that
\begin{equation} \label{eq:modes}
\left\lbrace
  \begin{array}{rl}
      \ppa u_k + \lambda_k u_k = f_k, & \\
      u_k(0) = v_k,  &\\
      u'_k(0) = b_k, & \text{if } \alpha \in (1,2].
      \end{array}
    \right.
\end{equation}
where $f_k = \left(f, \phi_k \right)$, $v_k = \left(v, \phi_k \right),$ and  $b_k = \left(b, \phi_k \right).$ Existence and uniqueness of solutions to \eqref{eq:modes} follow from standard theory for fractional-order differential equations \cite[Theorem 7.2]{Diethelm}.
Moreover, solutions of \eqref{eq:modes} may be represented as the superposition of the respective solution of the problem with forcing term and initial data equal to zero and the solution of the problem with vanishing forcing term. Namely, defining
$$
F_k (t) w = \int_0^t (t-r)^{\alpha-1} E_{\alpha,\alpha}(-\lambda_k (t-r)^\alpha) \, w(r) \, dr , 
$$
the solution of \eqref{eq:modes} may be written as
\begin{align}
u_k(t)  & =   F_k(t)f_k + v_k E_{\alpha, 1} (-\lambda_k t^\alpha) & (0 < \alpha \le 1),   \label{eq:sol_modes1} \\
u_k(t)  & =   F_k(t)f_k + v_k E_{\alpha, 1} (-\lambda_k t^\alpha) + b_k t E_{\alpha, 2} (-\lambda_k t^\alpha) \ & (1 < \alpha \le 2).  \label{eq:sol_modes2}
  \end{align}
For the particular value of $\alpha = 1$ the above expression yields the well-known formula
 $$ u_k (t) = \int_0^t e^{-\lambda_k(t-r)} f_k(r) \, dr \ + v_k e^{-\lambda_k t} ,$$
usually derived by the method of variation of parameters. 
Also, considering $\alpha = 2$, in virtue of the identities $E_{2,1}(z) = \cosh (\sqrt z)$ and
$E_{2,2}(z) = \frac{\sinh (\sqrt z)}{\sqrt z}$, expression \eqref{eq:sol_modes2} becomes
$$ u_k (t) = \frac{1}{\sqrt{\lambda_k}} \int_0^t \sin(\sqrt{\lambda_k}(t-r)) f_k(r) \, dr \ + v_k \cos (\sqrt{\lambda_k} t) + b_k \frac{\sin(\sqrt{\lambda_k} t)}{\sqrt{\lambda_k}} .$$

Summing the solutions for every eigenmode, we obtain the following result.
\begin{theorem} 
Let $\W$ be a bounded, smooth domain, $s \in (0,1)$ and $\alpha \in (0,2]$. Assume that \jp{$f \in  L^\infty(0,T;L^2(\W))$}, $v \in L^2(\W)$ and $b \in L^2(\W)$ are given.
Then, problem \eqref{eq:parabolic}, with boundary conditions \eqref{eq:boundary} (and  
\eqref{eq:boundary_wave} if $\alpha \in (1,2]$) admits a solution $u \in L^2(0,T; \widetilde H^{s}(\W))$ that  can be represented by \eqref{eq:sep_variables}. The modes $u_k$ are given, accordingly, by \eqref{eq:sol_modes1} if $\alpha \in (0,1]$ and  \eqref{eq:sol_modes2} if $\alpha \in (1,2]$.
\end{theorem}

\section{Regularity of solutions} \label{sec:regularity}
In this section we state some regularity results for solutions of the problems under consideration. We split the estimates according to whether the initial values or the forcing term are null. For the sake of brevity we omit the proofs; these can be obtained following the arguments of \cite{SakamotoYamamoto}, and are based on expansion \eqref{eq:sep_variables} and utilizing Lemmas \ref{lemma:der_mittag} and \ref{lemma:bound_mittag} to bound the corresponding coefficients. We also recall also that throughout this paper we are assuming that $\W$ is a domain with smooth boundary, so that Proposition \ref{prop:regHr} holds. According to that proposition, we fix the notation $\gamma := \min \{ s, 1/2 - \eps \}$, with $\eps > 0$ arbitrarily small.

\begin{theorem} \label{teo:reg_chico}
Let $0<\alpha \le 1$ and suppose that $f \equiv 0$. Let $u$,  given by
\eqref{eq:sep_variables},  be the solution of \eqref{eq:parabolic} with initial and boundary conditions according to \eqref{eq:boundary}. 
\begin{enumerate}
\item If $v \in L^2(\W)$, then $u \in C([0,T];L^2(\W))\cap C((0,T];\widetilde{H}^s(\W)\cap H^{s+\gamma}(\W))$ and $\ppa u \in C ((0,T]; L^2(\W))$. Moreover, there exists a constant $C>0$ such that
\begin{align*}
 & \| u \|_{C([0,T];L^2(\W))} \le C \| v \|_{L^2(\W)} , \\
& \| u(\cdot, t) \|_{H^{s+\gamma}(\W)} + \| \ppa u(\cdot, t) \|_{L^2(\W)} \le C t^{-\alpha} \| v \|_{L^2(\W)} . 
\end{align*} 
\item Assume that $v \in \widetilde{H}^s(\W)$. Then, $u \in L^2(0,T; \widetilde{H}^s(\W)\cap H^{s+\gamma}(\W))$, $\ppa u \in L^2(\W\times(0,T))$, and  the following estimate holds:
$$
\| u \|_{L^2(0,T; H^{s+\gamma}(\W))} + \| \ppa u \|_{L^2(\W\times(0,T))} \le C  \| v \|_{H^s(\rn)} .
$$
\item Furthermore, if $v \in \widetilde{H}^s(\W)$ is such that $(-\Delta)^s v \in L^2(\W)$, then $u\in C([0,T];\widetilde{H}^s(\W)\cap H^{s+\gamma}(\W))$, $\ppa u \in C ([0,T]; L^2(\W))$ and the bound
$$
\| u \|_{C([0,T];H^{s+\gamma}(\W))} + \| \ppa u \|_{C([0,T];L^2(\W))} \le C \| (-\Delta)^s v \|_{L^2(\W)} 
$$
is satisfied for some $C>0$ independent of $v$.
\end{enumerate}
\end{theorem}

Regularity estimates for the \emph{fractional diffusion} problem with a non-homogeneous right-hand side function $f$ are also attainable.
\begin{theorem}
 Let $0<\alpha \le 1$ and $v \equiv 0$. Consider $u$,  given by \eqref{eq:sep_variables}, the solution of \eqref{eq:parabolic} with homogeneous initial and boundary conditions. If \jp{$f \in  L^\infty(0,T;L^2(\W))$}, then $u \in L^2(0,T;\widetilde{H}^s(\W)\cap H^{s+\gamma}(\W))$, $\ppa u \in L^2(\W\times(0,T))$ and
$$
\| u \|_{L^2(0,T;H^{s+\gamma}(\W))} + \| \ppa u \|_{L^2(\W\times (0,T))} \le C \| f \|_{L^\infty(0,T;L^2(\W))} .
$$
\end{theorem}

Estimates for the \emph{fractional diffusion-wave} case are obtained similarly.

\begin{theorem}\label{teo:reg_grande} 
Let $1<\alpha \le 2$ and suppose that $f \equiv 0$. Let $u$,  given by
\eqref{eq:sep_variables},  be the solution of \eqref{eq:parabolic} with initial/boundary conditions \eqref{eq:boundary} and  \eqref{eq:boundary_wave}. 
\begin{enumerate}
\item Assume that $v \in L^2(\W)$ and $b \in L^2(\W)$. Then, $u \in C([0,T];L^2(\W))\cap C((0,T];\widetilde{H}^s(\W)\cap H^{s+\gamma}(\W))$ and $\ppa u \in C ((0,T]; L^2(\W))$. Moreover, there exists a constant $C>0$ such that
\begin{align*}
 & \| u \|_{C([0,T];L^2(\W))}  \le C \left( \| v \|_{L^2(\W)} + \| b \|_{L^2(\W)} \right) , \\
& \| u(\cdot, t) \|_{H^{s+\gamma}(\W)} + \| \ppa u(\cdot, t) \|_{L^2(\W)} \le C \left( t^{-\alpha} \| v \|_{L^2(\W)} + t^{1-\alpha} \| b \|_{L^2(\W)} \right)  . 
\end{align*} 
\item If $v \in \widetilde{H}^s(\W)$ and $b \in L^2(\W)$, then $\pp_t u \in C([0,T];H^{-s}(\W))$, and
$$
\| \pp_t u \|_{C([0,T]; H^{-s}(\W))} \le C (\| v \|_{H^s(\rn)} + \| b \|_{L^2(\W)}).
$$
\item Moreover, if $v \in \widetilde{H}^s(\W)$ is such that $(-\Delta)^s v \in L^2(\W)$ and $b \in \widetilde{H}^s(\W)$, then $u \in C([0,T]; \widetilde{H}^s(\W)\cap H^{s+\gamma}(\W)) \cap C^1([0,T];L^2(\W))$, $\ppa u \in C([0,T];L^2(\W))$, and  the following estimates hold:
\begin{align*}
& \| u \|_{C([0,T]; H^{s+\gamma}(\W))} + \| \ppa u \|_{C([0,T];L^2(\W))} \le C  \left( \| (-\Delta)^s v \|_{L^2(\W)} + \| b \|_{H^s(\rn)} \right), \\
&   \| u \|_{C^1([0,T]; L^2(\W))} \le C \left(\| (-\Delta)^s v \|_{L^2(\W)} + \| b \|_{L^2(\W)} \right) .
\end{align*}
\end{enumerate}
\end{theorem}

Finally, estimates for problems with non-null forcing term and $\alpha \in (1,2]$ have the following form.
\begin{theorem}
 Let $1<\alpha \le 2$, $v \equiv 0$ and $b \equiv 0$. Consider $u$,  given by
\eqref{eq:sep_variables}, be the solution of \eqref{eq:parabolic} with homogeneous initial and boundary conditions. If $f \in C([0,T]; L^2(\W))$ is such that $(-\Delta)^s f \in L^2(\W \times (0,T))$, then $u \in C([0,T];\widetilde{H}^s(\W)\cap H^{s+\gamma}(\W))$, $\ppa u \in C([0,T];L^2(\W))$ and
\begin{equation*}
\begin{split}
\| u \|_{C([0,T];H^{s+\gamma}(\W))} + & \| \ppa u \|_{C([0,T];L^2(\W))} \le  \\
& \le C \left( \| (-\Delta)^s f \|_{L^2(\W\times(0,T))} + \| f \|_{C([0,T];L^2(\W))} \right).
\end{split}
\end{equation*}
\end{theorem}

\section{Numerical approximations} \label{sec:discretization}
In this section we devise a discrete scheme to approximate \eqref{eq:parabolic}. To this end, standard Galerkin finite elements are utilized in the spatial discretization (following \cite{AcostaBorthagaray}) and a convolution quadrature is used for the time variable (following \cite{Jin}).

\subsection{Semi discrete scheme}
For an appropriate treatment, it is convenient to derive the numerical scheme in two steps. In first place we discretize in space, and afterwards in time.  We follow the ideas developed in \cite{Jin}, taking advantage of the fact that, from the theoretical point of view, minor changes are required to handle with the fractional Laplacian instead of its classical counterpart.  

Indeed, let $\mathcal{T}_{h}$ be a shape regular and quasi-uniform admissible simplicial mesh on $\Omega$, and let $X_{h} \subset \widetilde{H}^s(\Omega)$ be the piecewise linear finite element space associated with $\mathcal{T}_h$, namely,
$$\jp{
X_{h} := \{u_h \in C(\overline \Omega) \colon \restr{u_h}{T} \in \mathcal{P}^1 \ \forall T \in \mathcal{T}_{h}, \ u_h \big|_{\pp \W} = 0  \} .
}$$ 

The semidiscrete problem reads:  find $u_h \colon [0, T] \to X_h$ such that 
\begin{equation} 
\label{semiDiscreto}
\left\lbrace
  \begin{array}{rl}
      ( \ppa u_h  , w )  +  \langle u_h, w \rangle_s & =  \left( f, w \right), \quad \forall w \in X_h, \\
      u_h(0) & = v_h , \\
      u'_h(0) & = b_h, \text{ if } \alpha \in (1,2].
      \end{array}
\right.
\end{equation}
Here, $v_h = P_h v$, $b_h = P_h b$, and $P_h$ denotes the $L^2(\W)$ projection on $X_h$. 

Observe that, defining the discrete fractional Laplacian $A_{h}: X_h \rightarrow X_{h}$ as the unique operator that satisfies
$$ ( A_{h} w , v ) = \langle w , v \rangle_s, \text{  for all } w,v \in X_h, $$
and considering $f_h:=P_h f$, we may rewrite \eqref{semiDiscreto} as
\begin{equation} 
\left\lbrace
  \begin{array}{rl}
      \ppa u_h   +  A_h u_h & =  f_h, \\
      u_h(0) & = v_h , \\
      u'_h(0) & = b_h, \text{ if } \alpha \in (1,2].
      \end{array}
\label{semi}
\right.
\end{equation}

\subsection{Fully discrete scheme}
At this point, a suitable discretization of the Caputo differentiation operator is required to obtain a fully discrete scheme. To this end, we employ the convolution quadrature technique described by Lubich in \cite{Lub2, Lub1}, which allows us to derive discrete estimations of an integral which involve singular kernels.

Upon dividing $[0,T]$ uniformly with a time step size $\tau = T/N$, and letting $t = n \tau$ ($n \in \{ 1, \ldots, N\}$), by means of the convolution  quadrature rule we are able to estimate the Riemann-Liouville operator of a function $g$ by
\begin{equation}
\label{conv_dis}
\rl g(t) \approx  \sum_{j = 0}^n  \omega_j g(t - j\tau), 
\end{equation}
where the weights $\{\omega_j\}_{j \in \mathbb{N}_0 }$ are obtained as the coefficients of the power series 
$$
\left( \frac{1 - \xi}{\tau} \right)^{\alpha} = \sum^{\infty}_{j=0} \omega_{j}\xi^j .
$$

For the reader's convenience we give an overview of the main ideas of this technique in Appendix \ref{sec:CQ} and refer the reader, for example, to \cite{Lub2, Lub1} for further details.

 We are now able to suitably discretize the Caputo differentiation operator. To this end, we need to reformulate \eqref{eq:parabolic} using the Riemann-Liouville derivative instead of the Caputo one. It is well-known that these two operators are related by (see, for example, \cite[Theorem 3.1]{Diethelm})
$$
\ppa u(t) = \rl \left( u(t) - \sum_{k=0}^{\left\lfloor \alpha \right\rfloor} \frac{u^{(k)}(0)}{k!}t^k \right).
 $$
Thus, we rewrite \eqref{semi} for the \emph{fractional diffusion} case as
\begin{equation*} 
\left\lbrace
  \begin{array}{rl}
      \rl( u_h - v_h)  +  A_h u_h & =  f_h \\
      u_h(0) & = v_h, 
      \end{array}
\label{semi_rl}
\right.
\end{equation*}
and for the \emph{fractional diffusion-wave} case as
\begin{equation*} 
\left\lbrace
  \begin{array}{rl}
      \rl( u_h - v_h - tb_h)   +  A_h u_h & =  f_h \\
      u_h(0) & = v_h  \\
      u'_h(0) & = b_h.
      \end{array}
\label{semi_rl_ondas}
\right.
\end{equation*}

Replacing the Riemann-Liouville derivative by its discrete version given by \eqref{conv_dis}, and that we will denote by $\dtd^{\alpha}$, we formulate the fully discrete problem as: find $U^n_{h} \in X_h$, with $n = \{1,\ldots,N \}$, such that  
\begin{equation} 
\left\lbrace
  \begin{array}{rl}
      \dtd^{\alpha} U_h^n   +  A_h U_h^n & =  \dtd^{\alpha}v_h +  F_h^n \\
      U_h^0 & = v_h, 
      \end{array}
\label{full}
\right.
\end{equation}
or
\begin{equation} 
\left\lbrace
  \begin{array}{rl}
     \dtd^{\alpha} U_h^n   +  A_h U_h^n & =  \dtd^{\alpha}v_h + (\dtd^{\alpha}t)b_h +  F_h^n \\
      U_h^0 & = v_h, 
      \end{array}
\label{full_ondas}
\right.
\end{equation}
 for \emph{fractional diffusion} and \emph{fractional diffusion-wave} problems respectively, where $F^n_h = P_hf(t_n)$.

In order to obtain a better error estimation in the \emph{diffusion-wave} case, it is necessary to replace $F^n_h$ with a corrected term $G^n_h := \dtd \partial^{-1}_t f_h(t_n) $. See \cite{Jin} for further details.

For the sake of the reader's convenience, we conclude this section by giving the vectorial form of the fully discrete scheme. Let $\{ \varphi_i \}_{i=1,\ldots,\mathcal{N}}$ be the Lagrange nodal basis that generates $X_h$. Let $U^n$, $F^n$ and $G^n \in \mathbb{R}^{\mathcal{N}}$, $n=0,\ldots,N$ be such that $U^n_h = \sum_{i=1}^{\mathcal{N}} U^n_i\varphi_i$, $F^n_h = \sum_{i=1}^{\mathcal{N}} F^n_i\varphi_i$ and $G^n_h = \sum_{i=1}^{\mathcal{N}} G^n_i\varphi_i$, where $U_h^n$ denotes the solution of the fully discrete problem. 
Then we formulate \eqref{full} and \eqref{full_ondas}, respectively, in the following vectorial equations:
\begin{equation*} 
     M^{-1}\cdot(\w_0 M  +  K) \cdot U^n  =  \left(\sum_{j=0}^{n}\w_j \right)U^{0} - \sum^n_{j=1}\w_jU^{n-j} +  F^n
\label{full_vect}
\end{equation*}
and
\begin{equation*} 
  \begin{aligned}
     M^{-1}\cdot(\w_0 M  +  K) \cdot U^n  & =  \left(\sum_{j=0}^{n}\w_j \right)U^{0} + \left(\sum_{j=0}^{n}\w_j\tau(n-j) \right) v_h\\ &  - \sum^n_{j=1}\w_jU^{n-j} +  G^n. 
      \end{aligned}
\label{full_ondas_vect}
\end{equation*}
Above, $M,K \in \mathbb{R}^{\mathcal{N}\times\mathcal{N}}$ are the mass and stiffness matrices, respectively. Namely, $M_{i,j} = (\varphi_i,\varphi_j)$ and $K_{i,j} = \langle \varphi_i , \varphi_j \rangle_s$. 

The computation and assembly of the stiffness matrix in  dimension greater than one is not a trivial task. However, this problem for two-dimensional domains has been tackled in \cite{ABB}, where the authors provide a short MATLAB implementation.

There are several options to compute the coefficients $\{\w_j\}_{j \in \mathbb{N}_0 }$. Recalling that 
$$
\left(\frac{ 1 - \xi}{\tau} \right)^{\alpha} = \sum_{j=0}^{\infty} \w_j \xi^n, 
$$
 Fast Fourier Transform can be used for an efficient computation of $\{\w_j\}_{j \in \mathbb{N}_0 }$ (see \cite[Section 7.5]{Podlubny}) . 
Alternatively, a useful recursive expression is given also in \cite[formula (7.23)]{Podlubny}:  
\begin{equation*}
\w_0 = \tau^{-\alpha}, \quad \w_j = \left( 1 - \frac{\alpha + 1}{j}  \right) \w_{j-1}, \quad \forall j>0.
\end{equation*}  
For the numerical experiments we exhibit in Section \ref{sec:numerical} we have taken advantage of this identity.

\section{Error bounds}\label{sec:error}
This section shows error estimates for the numerical scheme discussed in Section \ref{sec:discretization}. The derivation of the error bounds can be carried out following the guidelines from \cite{Jin}. Most of these results can be extrapolated to our case. \jp{Details of the proofs in those cases where the generalization is not direct are provided in Appendix \ref{sec:Err}.}

\subsection{Error bounds for the semi-discrete scheme}
Here we focus only on the \emph{diffusion-wave} case ($1<\alpha<2$), where the generalization of the error bounds becomes more laborious. In this context, the error bounds are given by the generalization of two theorems. The first one, in the spirit of \cite[Theorem 3.2]{Jin}, provides an error estimation for the homogeneous case.

\begin{theorem} \label{teo:semi}
Let $1 < \alpha < 2$, $u$ be the solution of \eqref{eq:parabolic} with $v \in \widetilde H^q(\W)$, $b \in \widetilde H^r(\W)$ $q,r \in [0,2s]$, and $f = 0$; and let $u_h$ be the solution of \eqref{semiDiscreto} with $v_h = P_h v$, $b_h = P_h b$, and $f_h = 0$. Writing $e_h(t) = u(t) - u_h(t)$, there exists a positive constant $C = C(s,n)$ such that 
$$
\|e_h\|_{L^2(\W)} + h^{\gamma}|e_h|_{H^s(\rn)} \leq  C h^{2 \gamma} \left(t^{-\alpha \left( \frac{2s -  q}{2s} \right)} \|v\|_{H^q(\rn)}  +  t^{1-\alpha \left( \frac{2s -  r}{2s} \right)} \|b\|_{H^r(\rn)} \right).
$$
\end{theorem}

To complete the error estimate for the semi-discrete scheme, it still remains to analyze the case $v = 0, b = 0$ and $f \not = 0$. A proper generalization of  \cite[Theorem 3.2]{errorF} can be carried out following the guidelines outlined in that work.
\begin{theorem}\label{teo:semi_nonhom}
Let $1 < \alpha < 2$, $f \in L^{\infty}(0,T ; L^2(\W) )$, and let $u$ and $u_h$ be the solutions of \eqref{eq:parabolic} and \eqref{semiDiscreto} respectively, with $f_h = P_h f$, and all the initial data equal to zero. Then, there exists a positive constant $C = C(s,n)$ such that
$$
\|u - u_h \|_{L^2(\W)} + h^{\gamma}|u - u_h |_{H^s(\rn)} \leq  C h^{2 \gamma } |\log h|^2 \|f\|_{ L^{\infty}( [0,T] ; L^2(\W) )}.
$$ 
\end{theorem}

\subsection{Error bounds for the fully-discrete scheme}
Considering all the theory displayed up to this point, error estimates for the fully-discrete scheme can be derived in the same way as in \cite{Jin}. We refer the reader to that work for the details.

\begin{theorem} \label{teo:fully}
Let $u$ be the solution of problem \eqref{eq:parabolic} with $v \in \widetilde H^q(\W)$, $b \in \widetilde H^r(\W)$ $q,r \in [0,2s]$, and $f = 0$; and let $U^n_h$ be the solution of \eqref{full} or \eqref{full_ondas} with $v_h = P_h v$, $b_h = P_h b$, and $F^n_h = 0$. Then, there exists a positive constant $C = C(s,n)$ such that
\begin{itemize}
\item If $0<\alpha<1,$ then
$$
\|u(t_n) - U^n_{h} \|_{L^2(\W)} \leq C \left( t_n^{\alpha \left( \frac{q}{2s} \right) - 1  }\tau + t_n^{-\alpha \left( \frac{2s-  q}{2s} \right)  }  h^{s + \gamma}  \right)\|v\|_{H^q(\rn)}.
$$
\item If $1<\alpha<2,$ then
$$ \begin{aligned}
\|u(t_n) - U^n_{h} \|_{L^2(\W)} & \leq C \left( t_n^{\alpha \left( \frac{q}{2s} \right) - 1 }\tau + t_n^{-\alpha \left( \frac{2s-  q}{2s} \right)  }  h^{s + \gamma }  \right)\|v\|_{H^q(\rn)} \\
& + C \left( t_n^{\alpha \left( \frac{r}{2s} \right)  }\tau + t_n^{1 -\alpha \left( \frac{2s -  r}{2s} \right) }  h^{2 \gamma }  \right)\|b\|_{H^r(\rn)}.
\end{aligned} $$
\end{itemize}
\end{theorem}

\begin{remark} In the previous theorem --and in Theorem \ref{teo:semi} as well -- we wrote the orders of convergence in term of  various Sobolev norms of the data. For clarity, hypotheses in theorems \ref{teo:reg_chico} and \ref{teo:reg_grande} just involved either $L^2$ or $H^s$ norms of the data. For instance, assuming that $v \in \widetilde H ^s(\W)$ is such that $(-\Delta)^s v \in L^2(\W)$ and $b \in \widetilde H ^s(\W)$, the conclusions of Theorem \ref{teo:fully} read
\[ \begin{aligned}
\|u(t_n) - U^n_{h} \|_{L^2(\W)} & \leq C \left( t_n^{\alpha - 1 }\tau +  h^{s + \gamma}  \right)\|(-\Delta)^s v\|_{L^2(\W)} & \text{if } 0 < \alpha < 1,\\
\|u(t_n) - U^n_{h} \|_{L^2(\W)} & \leq C \left( t_n^{\alpha - 1 }\tau +  h^{s + \gamma}  \right)\|(-\Delta)^s v\|_{L^2(\W)} & \\
& + C \left( t_n^{\frac{\alpha}{2} }\tau + t_n^{1 -\frac{\alpha}{2}}  h^{2 \gamma }  \right)|b|_{H^s(\rn)} & \text{if } 1 < \alpha < 2.
\end{aligned} \]
\jp{
We emphasize that, as stated in Remark \ref{remark:mapping}, the identity $\|(-\Delta)^s v\|_{L^2(\W)} \le C | v |_{H^{2s}(\rn)}$ holds for all $v \in \widetilde H^{2s}(\W)$. 
}
\end{remark}

Finally, we state the order of convergence of the fully-discrete scheme for the problems with a non-null source term.
\begin{theorem}
Let $u$ be the solution of \eqref{eq:parabolic} with homogeneous initial data and with $f \in L^{\infty}(0,T;L^2(\W))$; and let $U^n_h$ be  the solution of \eqref{full} or \eqref{full_ondas} with  $f_h = P_h f$. Then, there exists a positive constant $C = C(s,n)$ such that
\begin{itemize}
\item For $0<\alpha<1$, if $\int^t_0 (t-s)^{\alpha-1}\|f'(s)\|_{L^2(\W)}ds < \infty$ for $t \in (0,T]$, then 
$$ \begin{aligned}
\|u(t_n) - U^n_{h} \|_{L^2(\W)}  \leq  C  & \Big(   h^{2 \gamma} \ell^2_h \|f\|_{ L^{\infty}( [0,T] ; L^2(\W) )} + t_n^{\alpha - 1}\tau \|f(0)\|_{L^2(\W)}  \\
& +\tau \int_0^{t_n}(t_n - s)^{\alpha - 1}\|f'(s)\|_{L^2(\W)} \Big).
\end{aligned} $$
\item If $1<\alpha<2,$ then
$$\|u(t_n) - U^n_{h} \|_{L^2(\W)} \leq C ( h^{2 \gamma } \ell^2_h + \tau) \|f\|_{ L^{\infty}( [0,T] ; L^2(\W) )}.$$
\end{itemize}
\end{theorem}

\section{Numerical experiments} \label{sec:numerical}
\jp{This section exhibits the results of numerical tests for discretizations of problems posed in one- and two-dimensional domains.}
Numerical solutions of \eqref{eq:parabolic} were obtained by applying the scheme described in Section \ref{sec:discretization}. The experiments in two-dimensional geometries were carried out with a code based on the one presented in \cite{ABB}.

\subsection{Explicit Solutions}
In \cite{ABB} it is shown how some families of non-trivial solutions for the fractional Poisson problem can be constructed.
For the sake of brevity, we refer the reader to that work for details. Here we summarize these results in order to be applied to the evolution equation in the cases in which $\Omega$ corresponds to {\emph a)} $(-1,1)\subset \R$ and, more generally, {\emph b)} $B(0,1)\subset \R^n$. 

Define for $n\ge 1$, the function $\w^s:\R^n\to \R,$ 
 $$\w^s(x) = (1 - |x|^2)_+ ^s. $$
 Then,  
 $$u(x):=\w^s(x) g_k^{(s)}(x)$$
 solves 
 $$
 \left\lbrace
  \begin{array}{rl}
      (-\Delta)^su =  f &  \mbox{in }  \Omega, \\
      u  = 0  &  \mbox{in }  \Omega^c,
      \end{array}
\right.
$$
 with $f(x)=\mu_s^k g_k^{(s)}(x)$, where in  case {\emph a)}
 $$
 \mu_s^k=\frac{\Gamma(2s + k + 1)}{k!} \qquad g_k^{(s)}(x):=C_{k}^{(s + 1/2)}(x),
 $$
and  in  case {\emph b)}
 $$
 \mu_s^k=\frac{2^{2s} \, \Gamma(1+s+k) \Gamma\left(\frac n2+s+k\right)}{k! \, \Gamma\left(\frac n2+k\right)} \qquad g_k^{(s)}(x):=P_k^{(s, \, n/2-1)} ( 2 |x|^2 - 1).
 $$
 Above, $C_{k}^{(s + 1/2)}$ and $P_k^{(s, \, n/2-1)}$ denote a Gegenbauer and  a Jacobi polynomial  \cite{Gegen}, respectively.
 
Next, let $h(t)$ be a function such that  $\ppa h(t)$ can be easily computed. By means of separation of variables we can construct explicit solutions of the fractional evolution problem of the form
$$
u(x,t) = h(t) \cdot \w^s(x) g_k^{(s)}(x).
$$ 

\subsection{Orders of convergence}
In order to confirm the predicted convergence rate, we show the results we obtained in three example problems: 
\begin{enumerate}
\item $u(x,t)= E_{\alpha,1}(-t^{\alpha}) \cdot \w^s(x) C^{(s)}_{3} (x) $, $\Omega = (-1,1)$;
\item $u(x,t)= \sin(t) \cdot \w^s(x) C^{(s)}_{3} (x)$, $\Omega = (-1,1)$;
\item $u(x,t)= E_{\alpha,1}(-t^{\alpha})  \cdot \w^s(x) P^{(s,0)}_{k}(2 | x|^2 -1) $, $\Omega = B(0,1) \subset \R^2.$
\end{enumerate}

For examples (a) and (b) we examine the time and spatial convergence over a fixed time $t = 0.1$. A fixed small time step is taking to see the spatial convergence and viceversa.  Our results are summarized in tables \ref{tabla_tiempo}, \ref{tabla_espacio_s075}, \ref{tabla_espacio} and \ref{tabla_2D}.

\begin{table}[ht]
\centering
\begin{tabular}{|c|c|c|c|c|c|c|}
\hline
Example & $\alpha \setminus\tau$ & 0.01     & 0.005    & 0.0025   & 0.001    & Rate \jp{(in $\tau$)}       \\ \hline
(a)     & 0.5              & 4.227e-3 & 2.105e-3 & 1.055e-3 & 4.493e-4 & 0.98 \\ \hline
(a)     & 1.5              & 2.512e-2 & 1.261e-2 & 6.349e-3 & 2.602e-3 & 0.99  \\ \hline
(b)     & 1.5              & 4.867e-3 & 2.402e-3 & 1.188e-3 & 5.362e-4 & 0.96 \\ \hline
\end{tabular}
\caption{The $L^2(\W)$ error at time $t = 0.1$ with $s = 0.75$ using a uniform mesh with size $h = 1/5000$. The expected convergence rate in $\tau$ is $1$. }
\label{tabla_tiempo}
\end{table}

\begin{table}[ht]
\centering
 \begin{tabular}{|c|c|c|c|c|c|c|}
\hline
Example & $\alpha \setminus$ mesh size $h$& $1/250$      & $1/500$     & $1/1000$     & $1/1500$   & Rate \jp{(in $h$)} \\ \hline
(a)     & 0.5                        & 8.837e-3 & 3.856e-3 & 1.781e-3 & 1.198e-3 & 1.12 \\ \hline
(a)     & 1.5                        & 9.978e-3 & 4.350e-3 & 1.967e-3 & 1.252e-3 & 1.16 \\ \hline
(b)     & 0.5                        & 1.162e-3 & 5.158e-4 & 4.010e-3 & 1.640e-4 & 1.09 \\ \hline
(b)     & 1.5                        & 8.453e-4 & 3.644e-4 & 1.632e-4 & 1.035e-4 & 1.17 \\ \hline
\end{tabular} 
\caption{The $L^2(\W)$ error at time $t = 0.1$ with $s = 0.75$ using $\tau = 1/5000$. The expected convergence rate in $h$ is 1.}
\label{tabla_espacio_s075}
\end{table}

\begin{table}[ht]
\centering
\begin{tabular}{|c|c|c|c|c|c|c|}
\hline
Example & $\alpha \setminus$ mesh size $h$ & $1/250$      & $1/500$     & $1/1000$     & $1/1500$   & Rate  \jp{(in $h$)} \\ \hline
(a)     & 0.5                        & 8.571e-2 & 4.999e-2 & 2.924e-2 & 2.139e-2 & 0.77 \\ \hline
(a)     & 1.5                        & 1.125e-1 & 6.596e-2 & 3.859e-2 & 2.818e-2 & 0.77 \\ \hline
(b)     & 0.5                        & 1.171e-2 & 6.845e-3 & 4.010e-3 & 2.937e-3 & 0.77 \\ \hline
(b)     & 1.5                        & 1.154e-2 & 6.811e-3 & 4.014e-3 & 2.943e-3 & 0.76 \\ \hline
\end{tabular}
\caption{The $L^2(\W)$ error at time $t = 0.1$ with $s = 0.25$ using $\tau = 1/5000$. The expected convergence rate in $h$ is 0.5.}
\label{tabla_espacio}
\end{table}

\begin{table}[ht]
\centering
\begin{tabular}{|c|c|c|}
\hline
Mesh size $h$ & $s= 0.25$ &  $s = 0.75$  \\ \hline
0.1    & 1.790e-1 &  5.673e-2    \\ \hline
0.05     & 1.102e-1 & 2.342e-2   \\ \hline
0.03    & 7.077e-2  & 1.054e-2 \\ \hline
0.02    & 5.206e-2  &  6.255e-3   \\ \hline
\end{tabular}
\caption{The $L^2(\W)$ error at time $t = 0.02$ with $s \in  \{0.25, 0.75 \}$ and $\alpha = 0.8$, using $\tau = 1/5000$ for example (c). The observed rates are $0.77$ and $1.38$, respectively. }
\label{tabla_2D}
\end{table}

The experimental orders of convergence (e.o.c.) are in agreement with the theory in the case $s>1/2$ while our numerical examples exhibit e.o.c. (in space) higher than those predicted if $s<1/2$ (see Tables \ref{tabla_espacio} and \ref{tabla_2D}). This behavior seems to be due to the fact that the extra regularity of the data present in our examples can not be exploited in our arguments; the actual solutions are more regular than what is predicted by Theorems \ref{teo:reg_chico} and \ref{teo:reg_grande}.

\subsection{Qualitative aspects} \label{ss:qualitative}
Finally, we present experiments that illustrate some qualitative effects of the fractional derivatives. In Figure \ref{ejemplo_grande_mas}, we fix $s=0.5$,  and show the evolution in time for different values of the parameter $\alpha$, ranging from \emph{fractional diffusion} to \emph{fractional diffusion-wave}. \jp{Memory effects are present for $\alpha = 0.5$, while the solution oscillates for $\alpha = 1.5$.} 

\begin{figure}
\centering
\begin{tabular}{|c|c|c|c|}
\hline
\subf{\includegraphics[width=25mm]{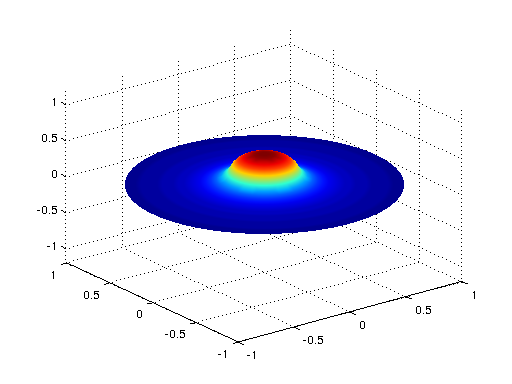}}
     {$\alpha = 0.5$, $t = 0.05$}
&
\subf{\includegraphics[width=25mm]{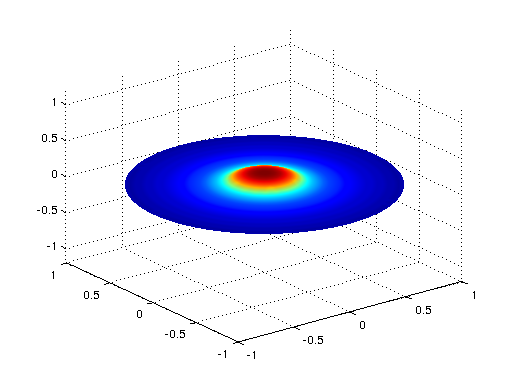}}
     {$\alpha = 0.5$, $t = 0.5$}
&
\subf{\includegraphics[width=25mm]{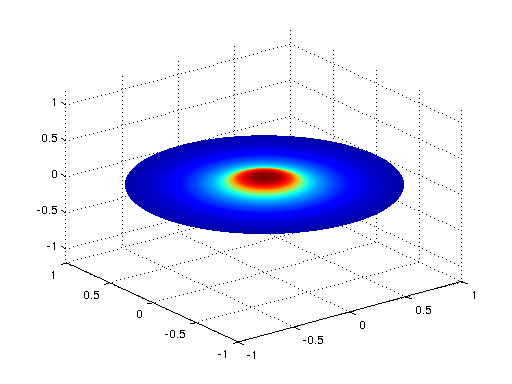}}
     {$\alpha = 0.5$, $t = 1.25$}
&
\subf{\includegraphics[width=25mm]{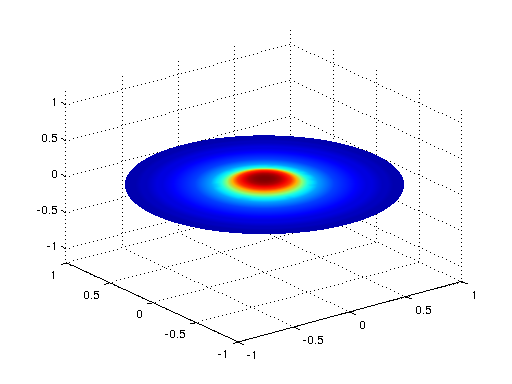}}
     {$\alpha = 0.5$, $t = 2$}
\\
\hline
\subf{\includegraphics[width=25mm]{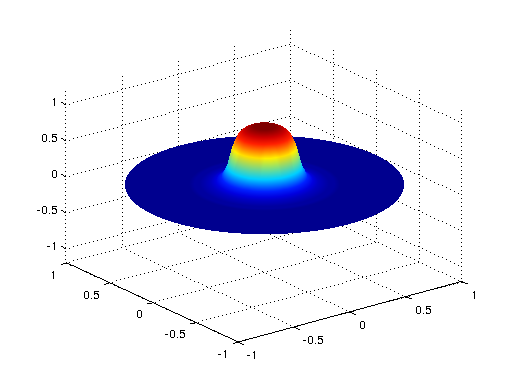}}
     {$\alpha = 1$, $t = 0.05$}
&
\subf{\includegraphics[width=25mm]{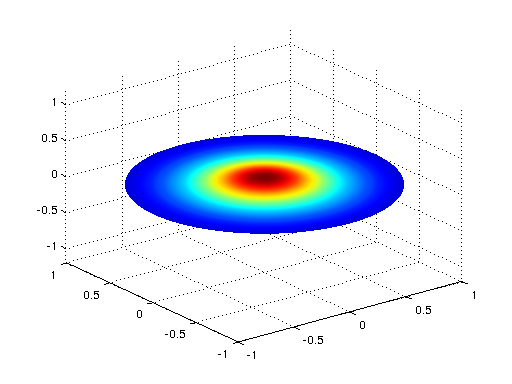}}
     {$\alpha = 1$, $t = 0.5$}
&
\subf{\includegraphics[width=25mm]{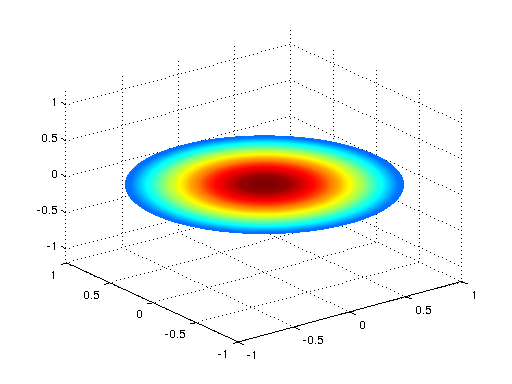}}
     {$\alpha = 1$, $t = 1.25$}
&
\subf{\includegraphics[width=25mm]{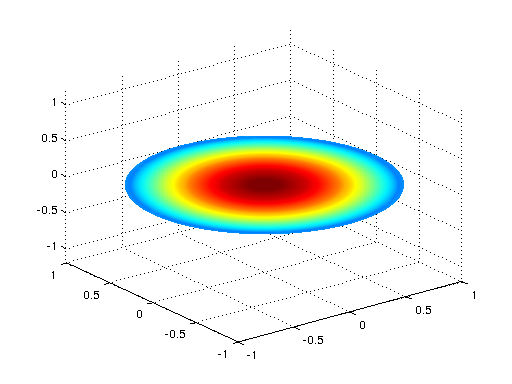}}
     {$\alpha = 1$, $t = 2$}
\\
\hline
\subf{\includegraphics[width=25mm]{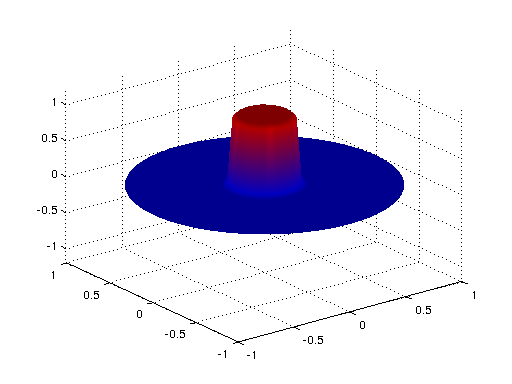}}
     {$\alpha = 1.75$, $t = 0.05$}
&
\subf{\includegraphics[width=25mm]{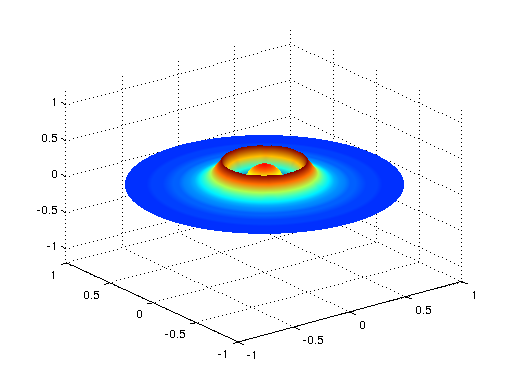}}
      {$\alpha = 1.75$, $t = 0.5$}
&
\subf{\includegraphics[width=25mm]{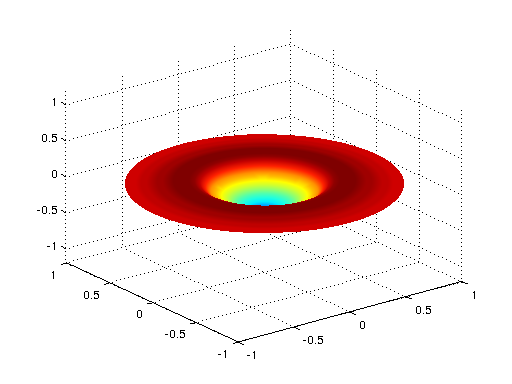}}
      {$\alpha = 1.75$, $t = 1.25$}
&
\subf{\includegraphics[width=25mm]{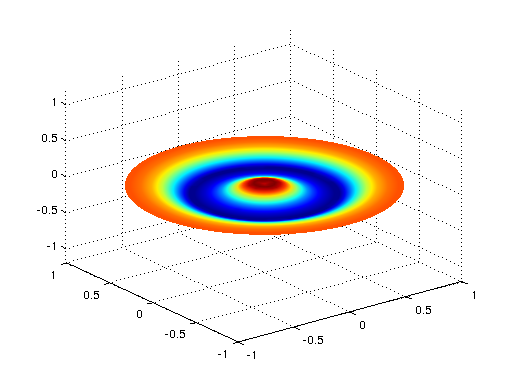}}
      {$\alpha = 1.75$, $t = 2$}
\\
\hline
\end{tabular}
\caption{In this example we set $\Omega = B(0,1)$, $s = 0.5$ and the initial conditions $v(x)=\chi_{B(0,r)}$ with $r=0.275$, and $b \equiv 0$ for $\alpha > 1$. The evolution in time is displayed for several values of $\alpha$.}
\label{ejemplo_grande_mas}
\end{figure}

Figure \ref{ejemplo_grande}, in turn, displays the effect of moving the parameters $\alpha$ and $s$ for a fixed time. \jp{It can be seen that increasing the spatial differentiation order $s$ leads to a faster spreading of the initial condition. Apparent differences can be noticed among the three different problems with $\alpha = 2s$ exhibited along the diagonal of the figure.}

\begin{figure}
\centering
\begin{tabular}{|c|c|c|c|}
\hline
\subf{\includegraphics[width=25mm]{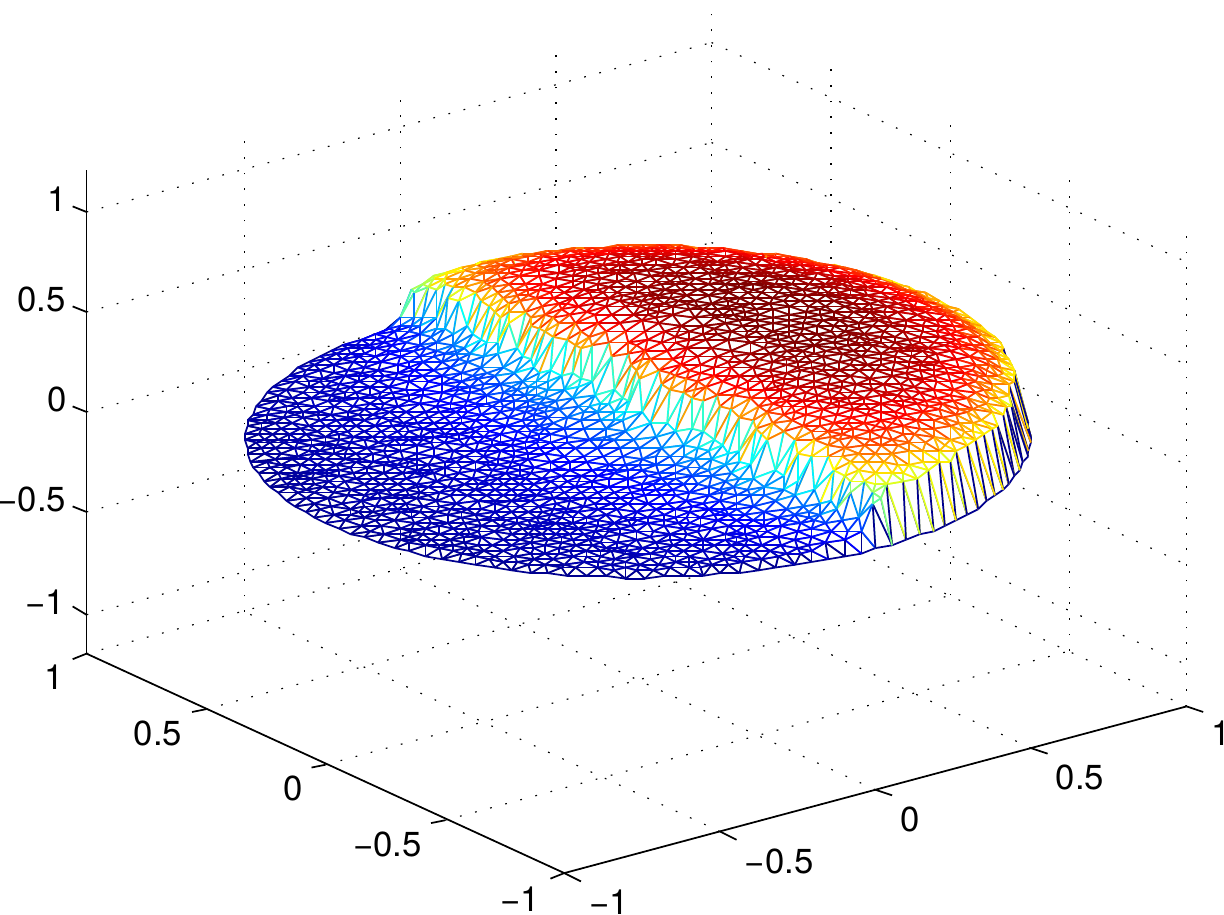}}
     {$\alpha = 0.5$, $s = 0.25$}
&
\subf{\includegraphics[width=25mm]{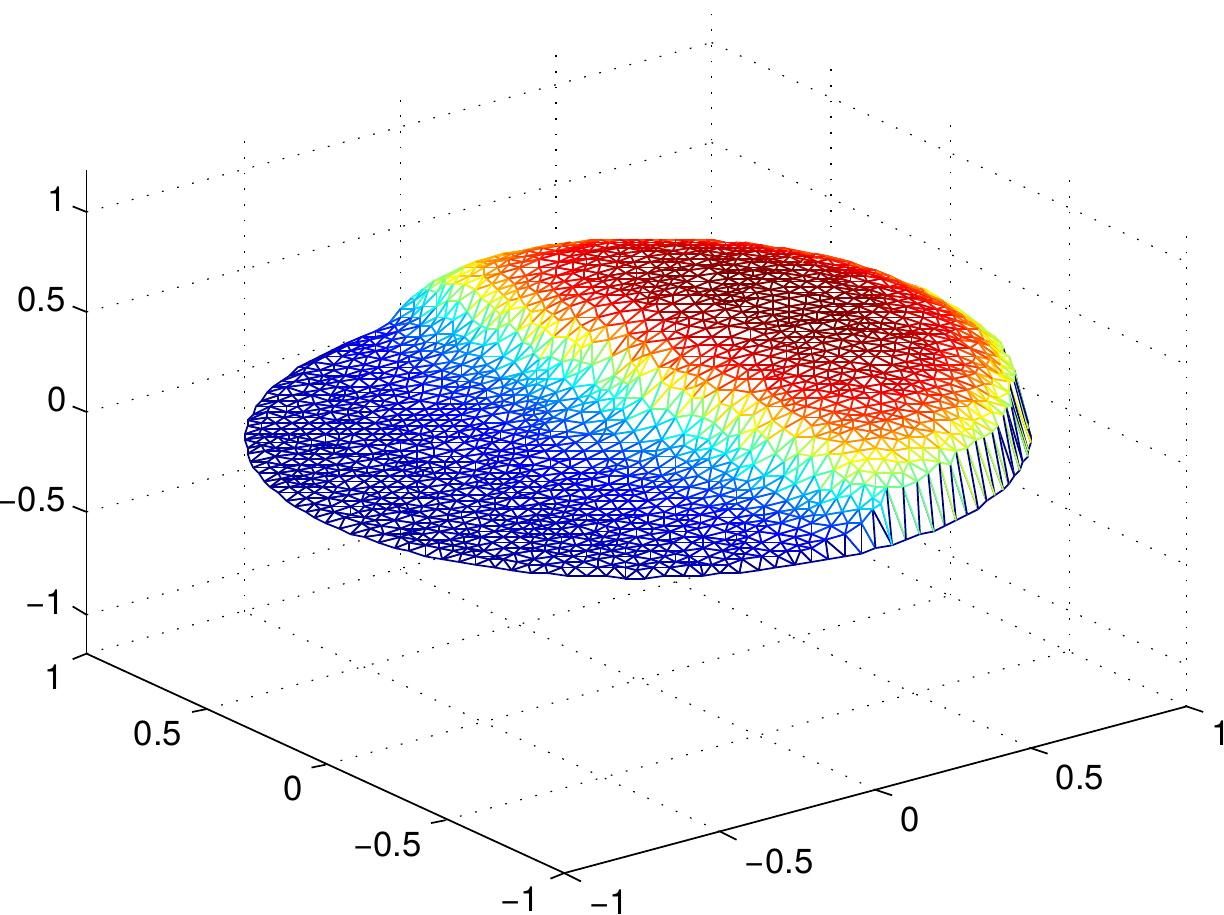}}
     {$\alpha = 1$, $s = 0.25$}
&
\subf{\includegraphics[width=25mm]{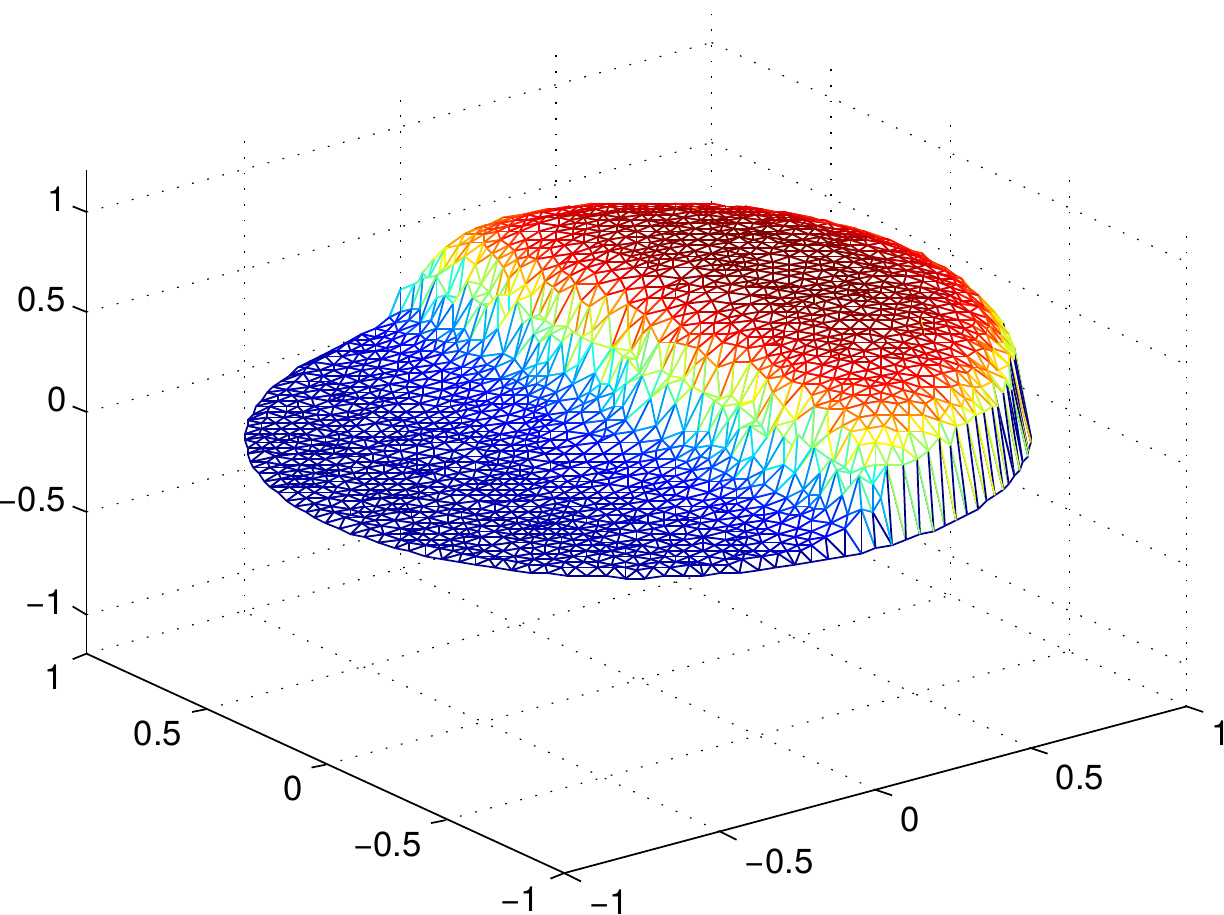}}
     {$\alpha = 1.5$, $s = 0.25$}
\\
\hline
\subf{\includegraphics[width=25mm]{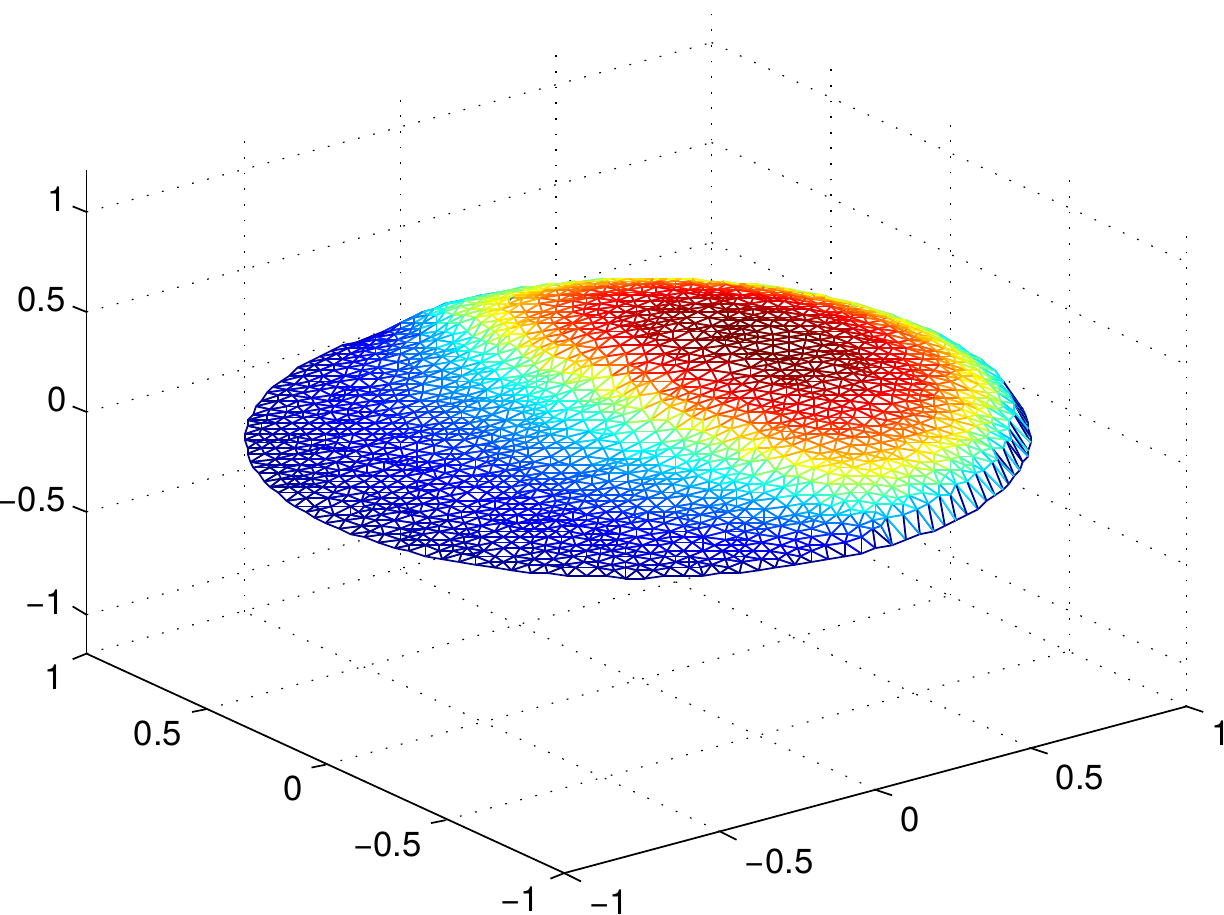}}
     {$\alpha = 0.5$, $s = 0.5$}
&
\subf{\includegraphics[width=25mm]{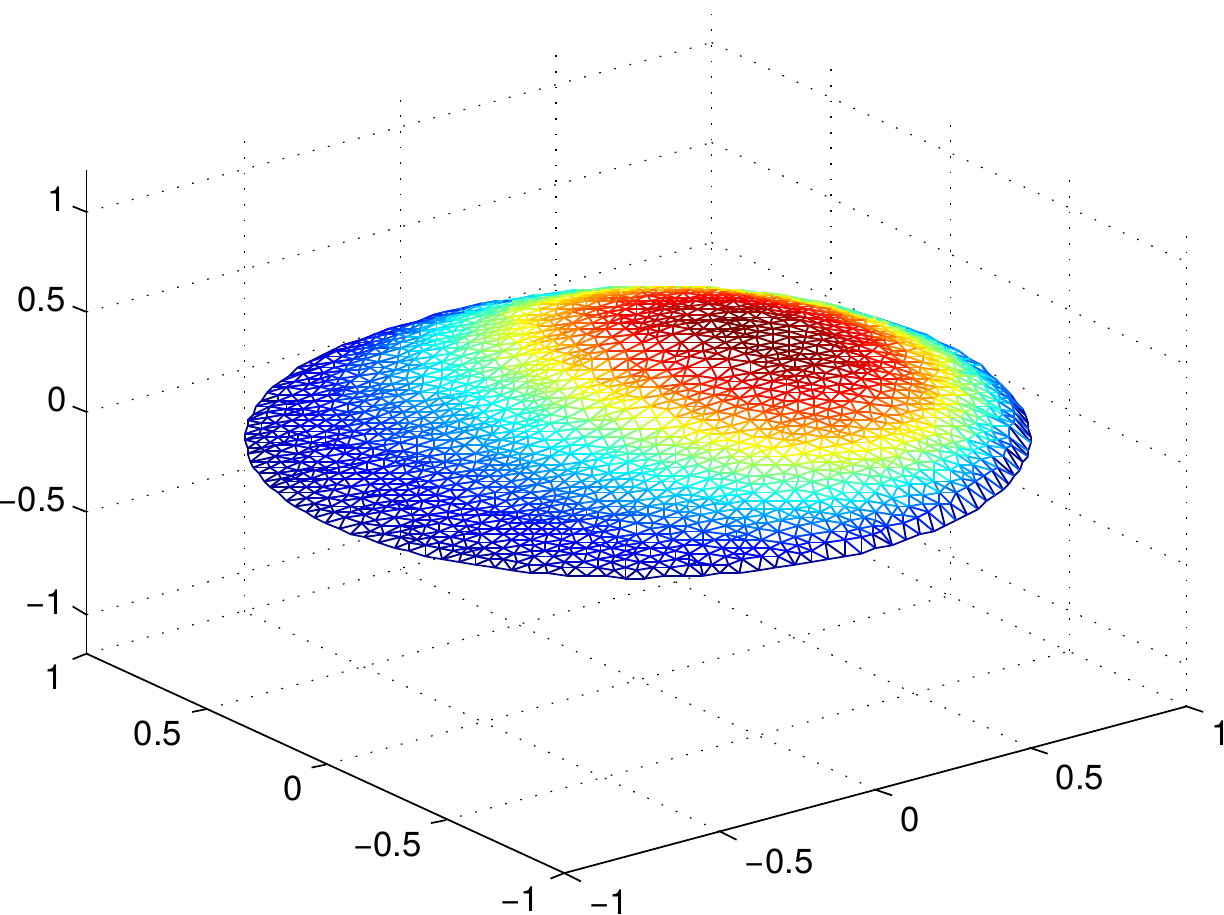}}
     {$\alpha = 1$, $s = 0.5$}
&
\subf{\includegraphics[width=25mm]{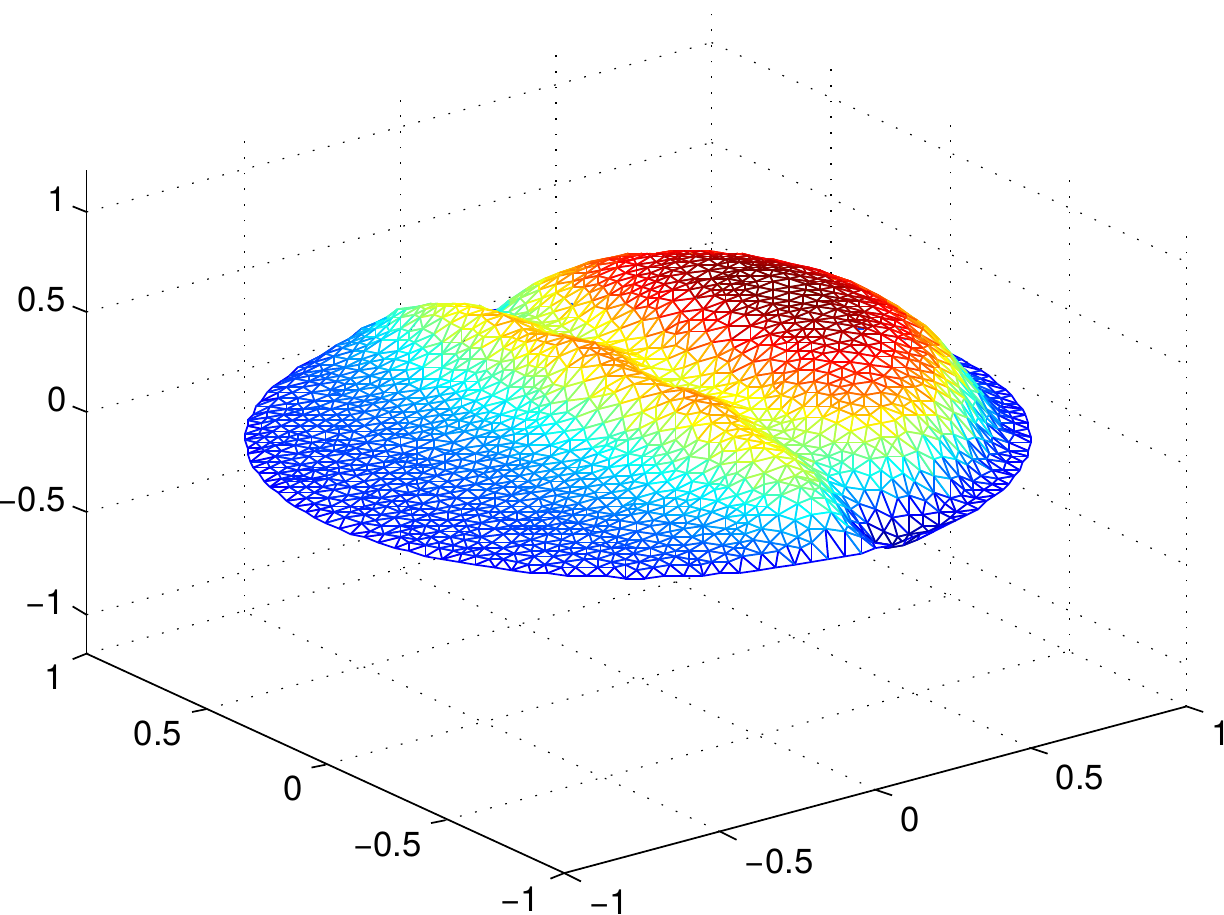}}
     {$\alpha = 1.5$, $s = 0.5$}
\\
\hline
\subf{\includegraphics[width=25mm]{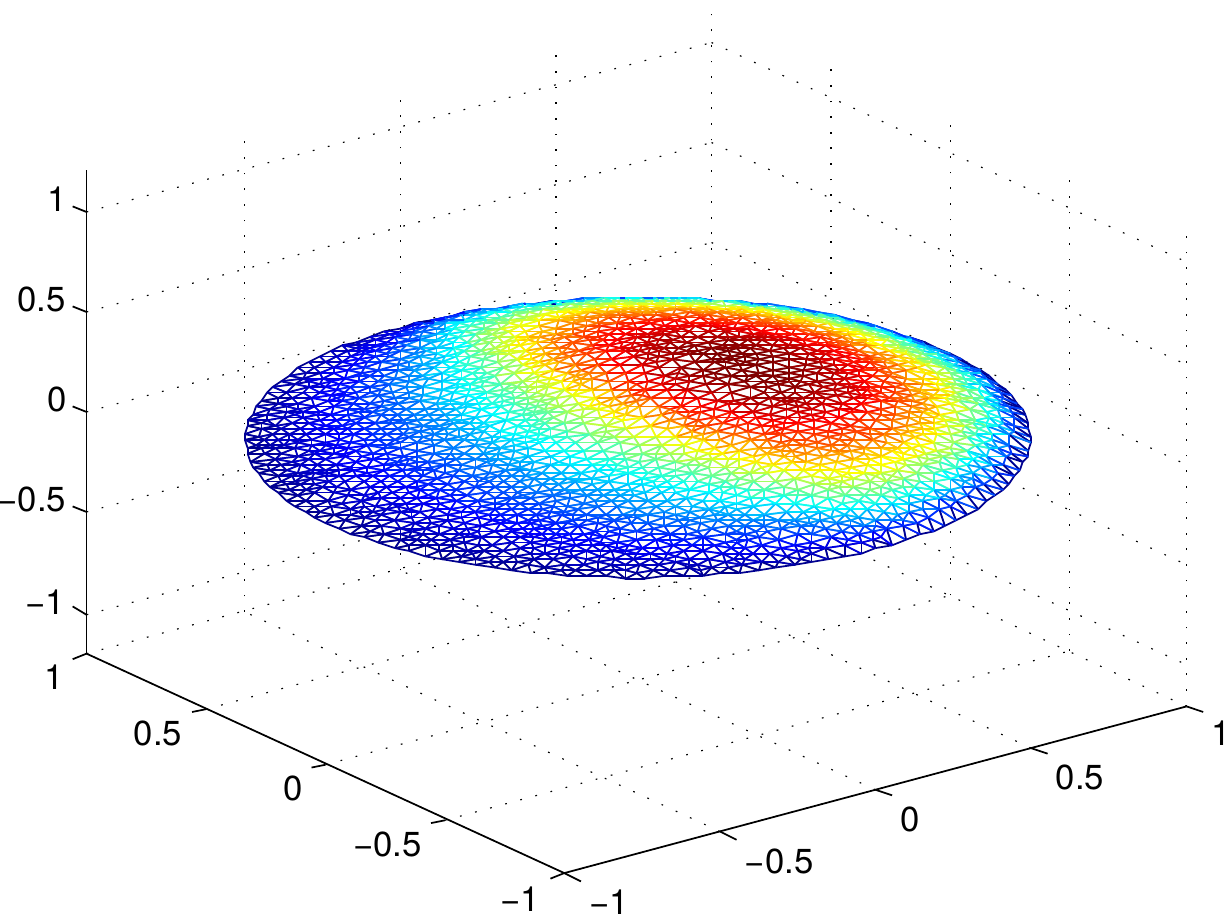}}
     {$\alpha = 0.5$, $s = 0.75$}
&
\subf{\includegraphics[width=25mm]{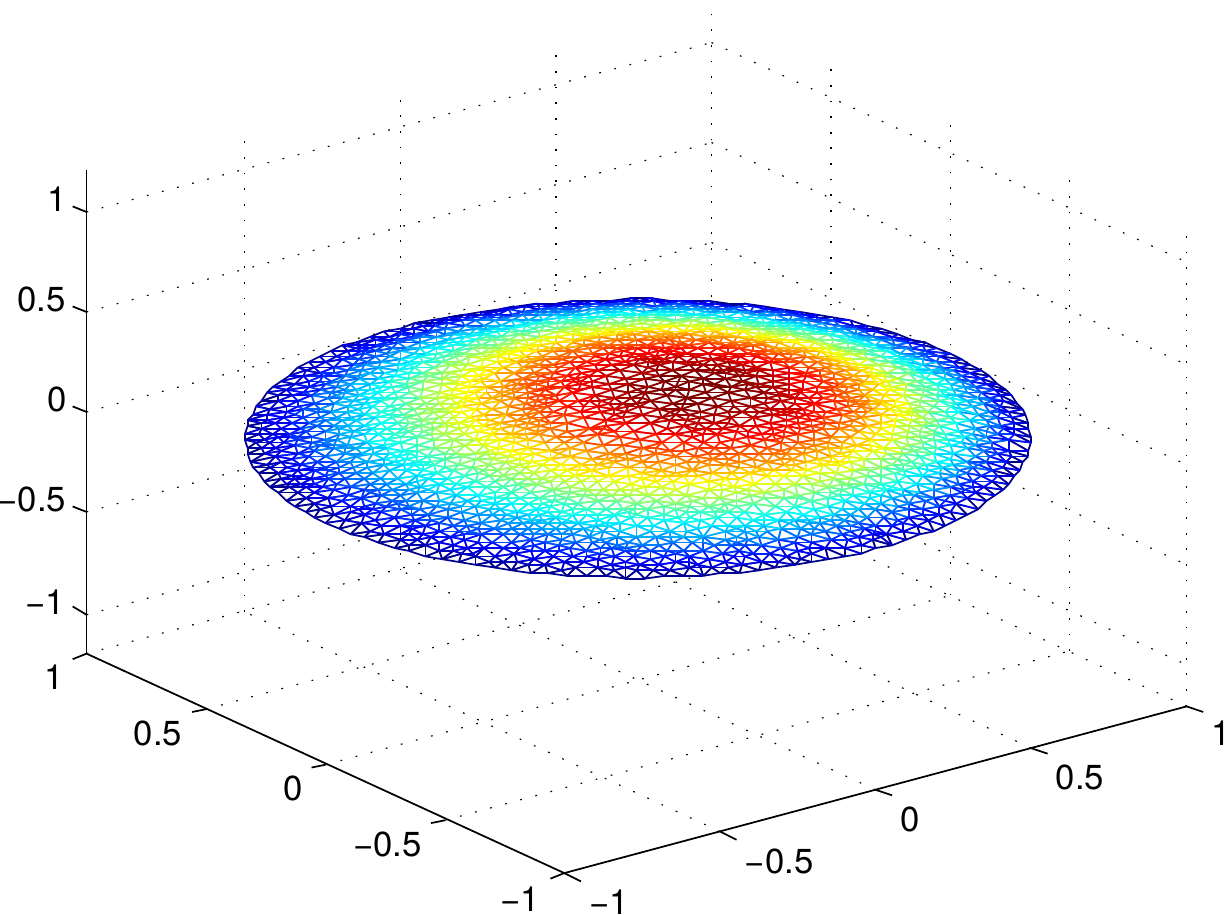}}
      {$\alpha = 1$, $s = 0.75$}
&
\subf{\includegraphics[width=25mm]{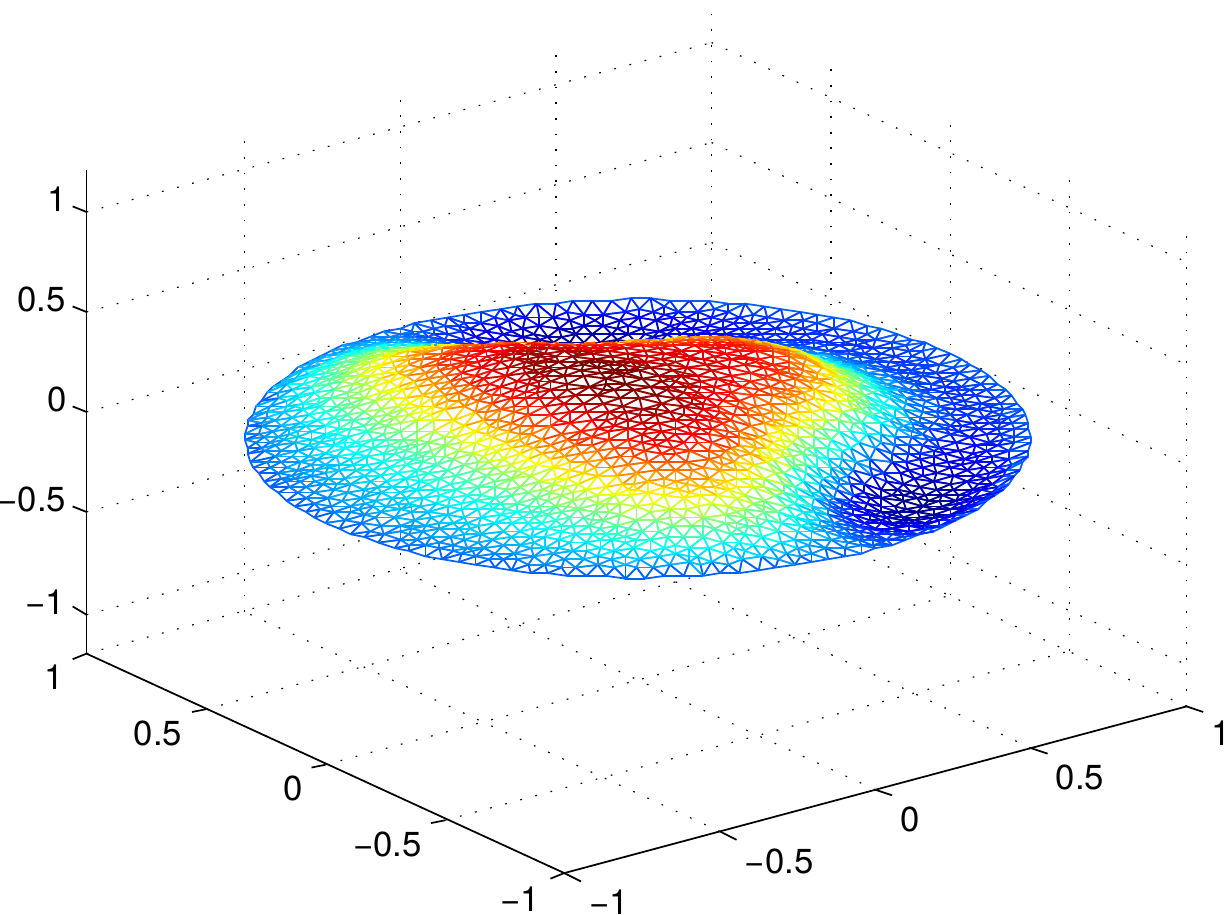}}
      {$\alpha = 1.5$, $s = 0.75$}
\\
\hline
\end{tabular}
\caption{Effect of considering various $\alpha$ and $s$ at time $t = 0.05$. In this example, $\Omega = B(0,1)$ and  the initial data are $v(x,y)=\chi_{\{y > 0 \}}(x,y)$ and $b \equiv 0$ for $\alpha > 1$.
} \label{ejemplo_grande}
\end{figure}

Our last example, in Figure \ref{ejemplo_delta}, exhibits the persistence of a singularity along the time due to the \emph{memory} induced by the fractional in-time derivative. \jp{In that experiment, we have set $\alpha = 0.99$ and $s = 0.9$. Notice that the solution vanishes to $0$ as time increases, but even though the differentiation parameters are both close to $1$, which corresponds to the classical heat equation, the singular behavior of the initial condition persists in time.}

\begin{figure}
\centering
\begin{tabular}{|c|c|}
\hline
\subf{\includegraphics[width=25mm]{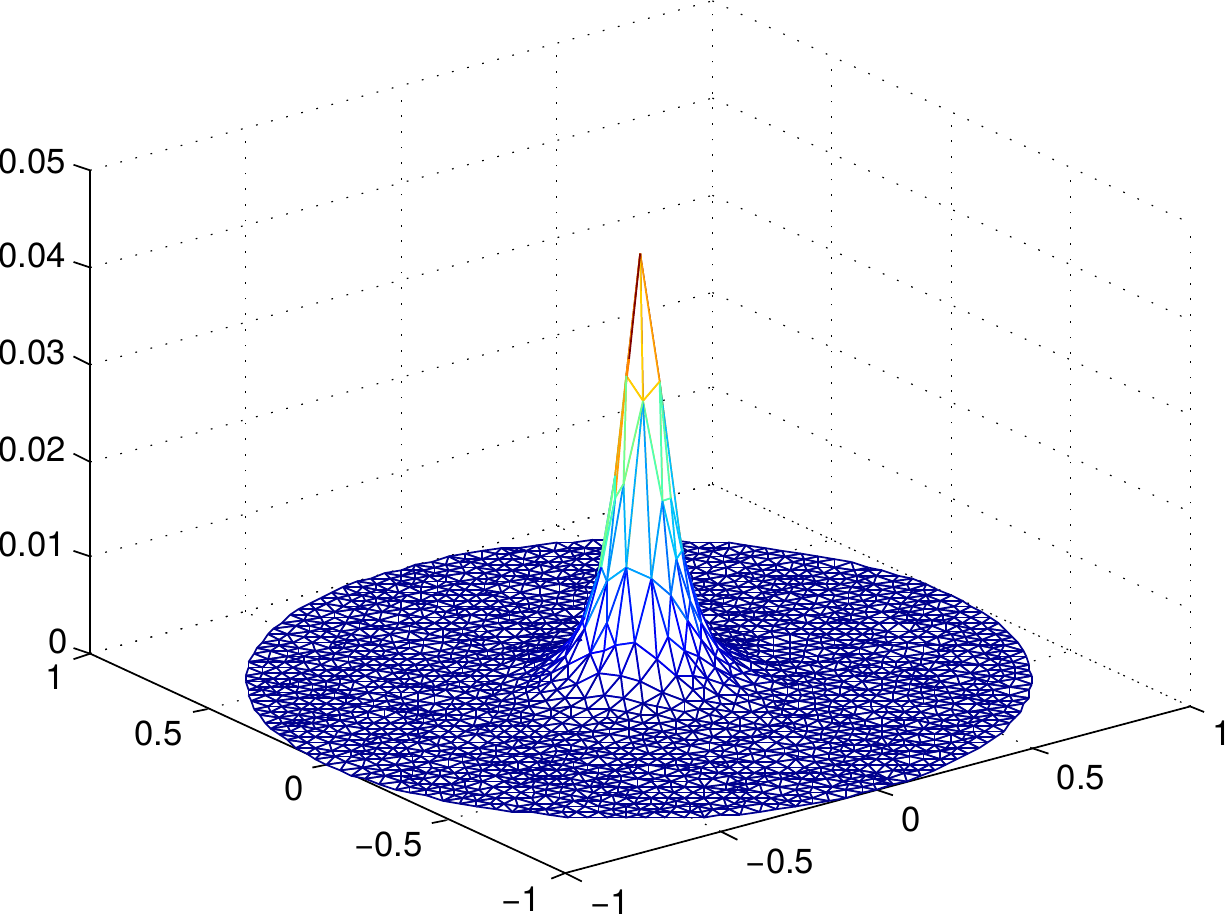}}
     {Solution at $t = 0.015$}
&
\subf{\includegraphics[width=25mm]{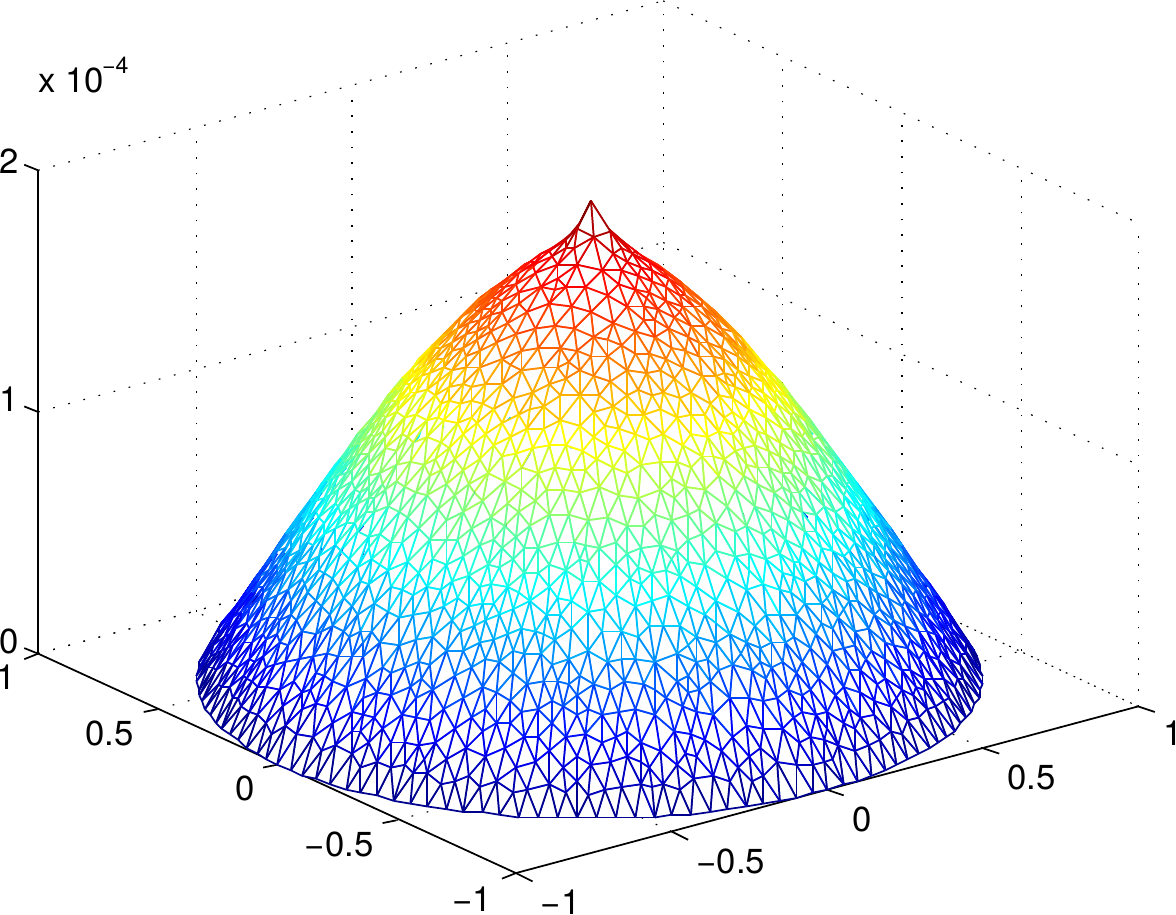}}
     {Solution at $t = 0.55$}
\\
\hline
\subf{\includegraphics[width=25mm]{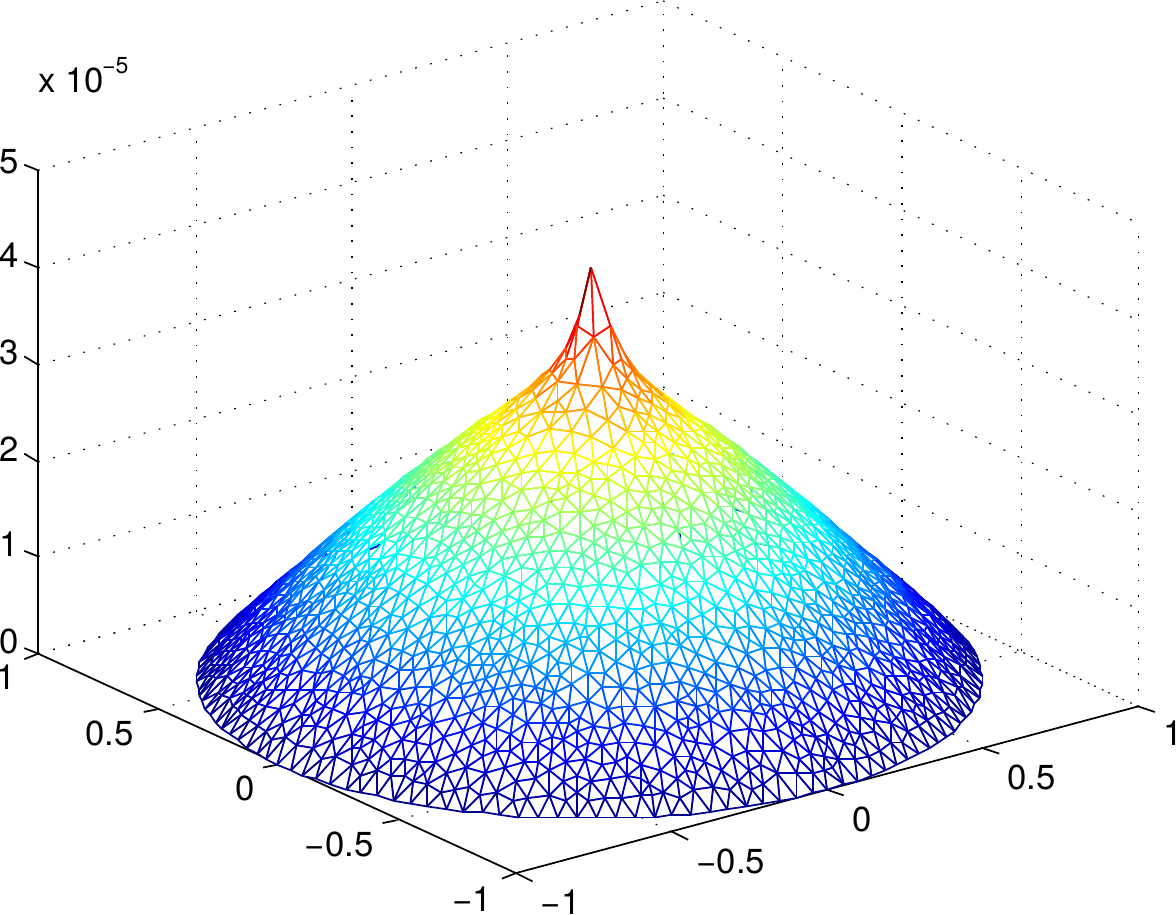}}
     {Solution at $t = 1$}
&
\subf{\includegraphics[width=25mm]{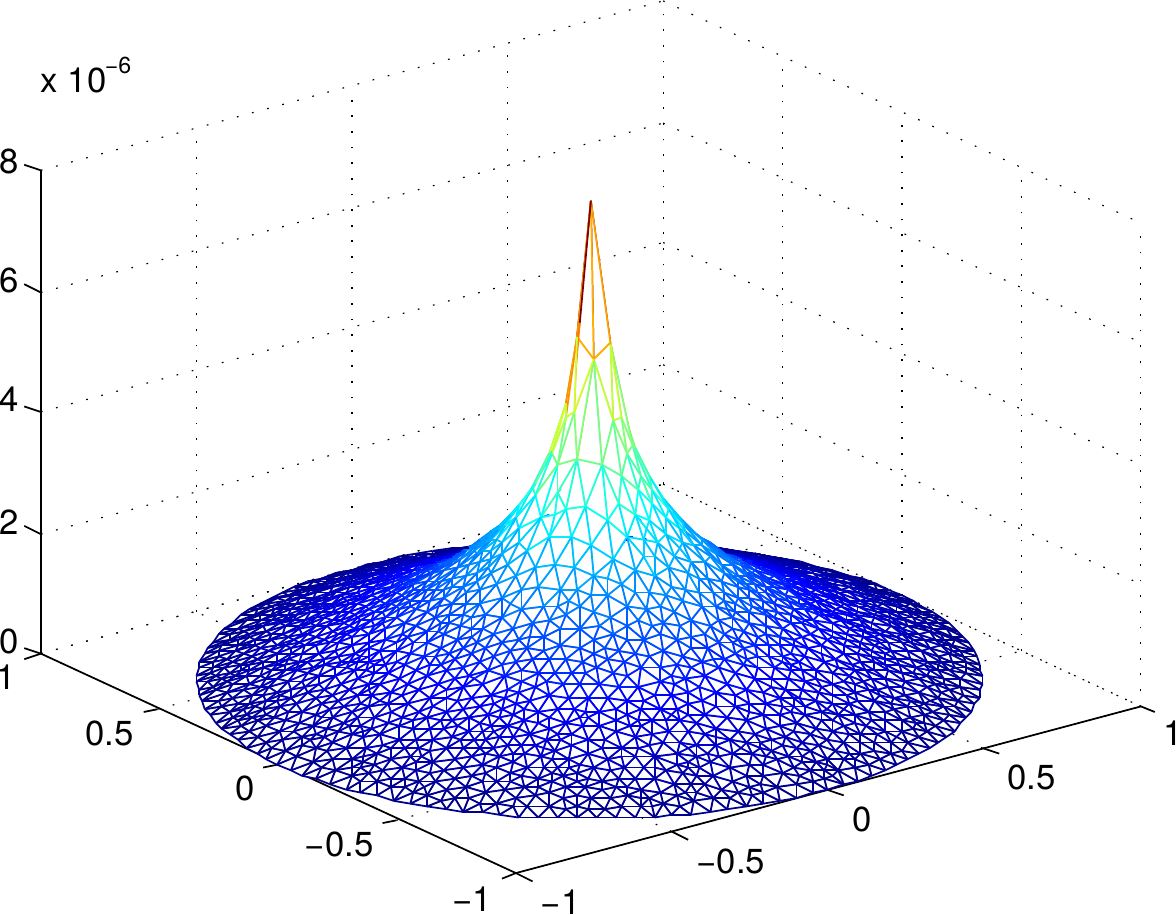}}
     {Solution at $t = 2.5$}
\\
\hline
\end{tabular}
\caption{The effect of a fractional derivative in time: observe the persistence of the singularity along the time, even for $\alpha\sim 1$. In this example we set $\Omega = B(0,1)$,  $\alpha = 0.99$, $s = 0.9$ and $v(0,0)=1$ and $v(x) = 0$ for any other node of the mesh.}
\label{ejemplo_delta}
\end{figure}

\subsection*{Acknowledgments} The authors thank Prof. Bangti Jin for pointing out reference \cite{Jin}.

\appendix

\section{Convolution Quadrature Rule}\label{sec:CQ}
The aim of this appendix is to describe a numerical approximation technique for convolutions that plays an important role in the assemblage of the numerical scheme we propose. 

Dividing $[0,T]$ uniformly with a time step size $\tau = T/N$, and letting $t = n \tau$ ($n \in \{ 1, \ldots, N\}$), we seek for a numerical approximation of the convolution integral
\begin{equation}
k*g(t)=\int^t_{0}k(r)g(t-r) \, dr
\label{conv}
\end{equation}
by means of a finite sum
$
\sum_{j = 0}^n  \omega_j g(t - j\tau).
$
The weights $\{\omega_j\}_{j \in \mathbb{N}_0 }$ are obtained as the coefficients of the power series 
$$
K\left( \frac{\delta(\xi)}{\tau} \right) = \sum^{\infty}_{j=0} \omega_{j}\xi^j ,
$$
where $K$ denotes the Laplace transform of the kernel $k$, and $\delta(\xi)$ is the quotient of the generating polynomials of a linear multistep method.    

To obtain the weights $\omega_j$, suppose that we extend the kernel $k$ by zero over $r \leq 0$ and that for all $r>0$ it satisfes 
\begin{equation}
|k(r)| \leq Cr^{\mu - 1}e^{cr},
\label{condicion_mu_kernel}
\end{equation}
for some $c,\mu > 0$. Then,  the inversion formula
\begin{equation}
k(r) = \frac{1}{2\pi i} \int_{\Gamma}K(z)e^{zr}dz
\label{inversion}
\end{equation}       
holds, where $\Gamma$ is a contour lying in the sector of analyticity of $K$, parallel to its boundary and oriented with an increasing imaginary part. Furthermore, defining $\Sigma_{\theta}:=\{ z \in \mathbb{C} : |\text{arg}(z)| \leq \theta\}$, $\theta \in (\pi/2,\pi)$, it holds that $K$ is analytic in $\Sigma_{\theta}$ and satisfies 
\begin{equation}
|K(z)| \leq C|z|^{-\mu} \quad \forall z \in \Sigma_{\theta}.
\label{condicion_mu}
\end{equation}
This condition is in turn equivalent to \eqref{condicion_mu_kernel}.

Replacing \eqref{inversion} in \eqref{conv} and switching the order of integration gives
\begin{equation}
\int^t_{0}k(r)g(t-r)dr =  \frac{1}{2\pi i} \int_{\Gamma} K(z) \int^{t}_{0}e^{zr}g(t-r)\, dr \, dz. 
\label{conv_inversion}
\end{equation}
Since the inner integral in the right-hand side is the solution of the ordinary differential equation $y' = zy + g$, with $y(0)=0$, we can obtain a numerical estimation by using some multistep method. For simplicity, suppose we utilize Backward Euler discretization (BE), that gives the scheme
$$\frac{y_{n} - y_{n-1}}{\tau} =  zy_n + g_n. $$
Multiplying by $\xi^n$ both sides of the equality, and summing over $n$, we obtain
\begin{equation}
\frac{(1-\xi)}{\tau} \mathbf{y}(\xi) =  z\mathbf{y}(\xi) + \mathbf{g}(\xi),
\label{series}
\end{equation}
where $\mathbf{y}(\xi) :=\sum^{\infty}_{n=0}y_n\xi^n$ and $\mathbf{g}(\xi):=\sum^{\infty}_{n=0}g_n\xi^n$. Defining $\delta(\xi):=(1-\xi)$, from \eqref{series} we deduce
$$
\mathbf{y}(\xi)=\left(\frac{\delta(\xi)}{\tau}-z\right)^{-1}\mathbf{g}(\xi). 
$$
Thus, the numerical approximation of $y$ at time $n\tau$ is given by the $n$-th coefficient of the power series $\left(\frac{\delta(\xi)}{\tau}-z\right)^{-1}\mathbf{g}(\xi)$.

In order to obtain the desired numerical approximation of \eqref{conv} we utilize the former expression,  fix $\xi$ and integrate in $z$ the right hand side in \eqref{conv_inversion}. Using Cauchy's integral formula gives
\begin{equation*}
\frac{1}{2\pi i}\int_{\Gamma}K(z)\left(\frac{\delta(\xi)}{\tau}-z\right)^{-1}\mathbf{g}(\xi) \, dz=K \left(\frac{\delta(\xi)}{\tau}\right) \cdot \mathbf{g}(\xi). 
\end{equation*}
Therefore, the numerical approximation of \eqref{conv} at $t=n\tau$ is given by the $n$-th coefficient of the power series $K \left(\frac{\delta(\xi)}{\tau}\right) \cdot \mathbf{g}(\xi)$. Finally, noticing that the coefficients of the series are the Cauchy product of the sequences $\{\w_n\}_{n \in \mathbb{N}_0}$ and $\{g(n\tau)\}_{n \in \mathbb{N}_0}$, where $\{\w_n\}$ are the coefficients of the power series expansion of $K \left(\frac{\delta(\xi)}{\tau}\right)$, we obtain an expression for the weights in \eqref{conv_dis}.

Given a complex valued function $K$, analytic in $\Sigma_{\theta}$ and satisfying \eqref{condicion_mu}, we use the transfer function notation for \eqref{conv},
\begin{equation*}
K(z)g(t) := k*g(t) = \int^t_{0}k(r)g(t-r)\, dr,
\label{transfer}
\end{equation*}
where $k$ is given by \eqref{inversion}, and the notation
\begin{equation*}
K\left(\frac{\delta(\xi)}{\tau}\right)g(t) := \sum_{j = 0}^n \omega_j g(t - j\tau)
\label{transfer_dis}
\end{equation*}
for the discrete approximation.

Next, we generalize the definition of the Convolution Quadrature Rule to operators that satisfy \eqref{condicion_mu} with a negative value of $\mu$. Indeed, let $m$ be a positive integer such that $\mu + m>0$, setting $\tilde{K}(z) := z^{-m}K(z)$ we define 
\begin{equation*}
K(z)g(t) := \frac{\pp^m}{\pp t^m} \tilde{k}*g(t) =  \frac{\pp^m}{\pp t^m} \int^t_{0}\tilde{k}(r)g(t-r)\, dr,
\label{generalizacion}
\end{equation*}
with $\tilde{k}$ the kernel associated with $\tilde{K}$. All the results and estimates that are achieved in the former case are still true upon this generalization (see \cite[Section 5]{Lub2}). 
This is convenient because we are interested in the particular case of $K(z)=z^{\alpha}$, that delivers
\begin{equation*}
z^{\alpha}g(t) := \frac{\pp^m}{\pp t^m} \int^t_{0}\frac{1}{r^{\alpha - m + 1}}g(t-r) \, dr = \rl g(t).
\end{equation*}
with $m$ a positive integer such that $m-1\le \alpha < m$. Considering this, we set the notation $K(\partial_t):=K(z)$ and $K( \dtd) := K(\delta(\xi)/\tau)$.

\section{Error analysis of the semi-discrete scheme}\label{sec:Err}
We first set an integral representation of $u$ for the homogeneous case $f = 0$.
Define the sector $\Sigma_{\theta} := \{ z \in \mathbb{C} \text{ such that } |\arg(u)|<\theta \}$, then $u(t) : [0,T] \rightarrow L^2(\Omega)$ can be analytically extended to $\Sigma_{\pi/2}$ (see \cite[Theorem 2.3]{SakamotoYamamoto}). Applying the Laplace transform in \eqref{eq:parabolic} we obtain
$$
z^{\alpha} \hat{u}(z) + A\hat{u}(z) = z^{\alpha - 1}v +  z^{\alpha - 2}b, 
$$
where $A$ is the fractional Laplacian with homogeneous Dirichlet conditions. Therefore, via the Laplace inversion formula, we write the integral representation
\begin{equation}
\label{formaIntegral}
u(t) = \frac{1}{2\pi i} \int_{\Gamma_{\theta,\delta}} e^{zt}(z^{\alpha}I + A)^{-1}( z^{\alpha - 1}v +  z^{\alpha - 2}b) \, dz,
\end{equation} 
where $\Gamma_{\theta,\delta} = \{ z \in \mathbb{C} : |z| = \delta, |\arg(z)|\leq\theta \} \cup \{ z \in \mathbb{C} : z = re^{\pm i \theta}, r \geq \delta \}$. 

If we choose $\theta$ such that $\pi/2 < \theta < \min \{ \pi,\pi/\alpha \}$, then $z^{\alpha} \in \Sigma_{\theta'}$ with $\theta' = \alpha \theta$ for all $z \in \Sigma_{\theta}.$ Considering $\theta$ in this mode, there exists a constant $C$ only depending on $\theta$ and $\alpha$ such that
\begin{equation*}
\label{semigrupo}
\|(z^{\alpha} + A)^{-1})\|_{L^2(\W)} \leq C|z|^{-\alpha}.
\end{equation*}

As in \eqref{formaIntegral}, we can write an analogous expression for $u_h$,
\begin{equation*}
\label{formaIntegralDiscreta}
u_h(t) = \frac{1}{2\pi i} \int_{\Gamma_{\theta,\delta}} e^{zt}(z^{\alpha}I + A_h)^{-1}( z^{\alpha - 1}v_h +  z^{\alpha - 2}b_h) \, dz.
\end{equation*} 
     
The following technical result can be proved analogously to \jp{\cite[Lemma 7.1]{FujitaSuzuki}.}
\begin{lemma}
\label{lemaTecnico}
Let $\varphi \in  \widetilde{H}^s(\W)$, and $z \in \Sigma_{\theta}$ with $\pi/2 < \theta < \min \{ \pi,\pi/\alpha \}$. Then there exists a positive constant $c(\theta)$ such that
$$
\left|z^{\alpha}\right| \| \varphi \|_{L^2(\W)}^2 +  | \varphi |_{H^s(\rn)}^2 \leq c\left| z^{\alpha}\| \varphi \|_{L^2(\W)}^2 +  | \varphi |_{H^s(\rn)}^2  \right|.
$$ 
\end{lemma}

The next lemma sets an error estimate between $(z^{\alpha}I + A  )^{-1}f$ and its discrete approximation $(z^{\alpha}I + A_h  )^{-1}P_h f$, analogous to \cite[Lemma 3.4]{Bazhlekova}. 

\begin{lemma}
\label{lemaError}
Let $f \in L^{2}(\W)$, $z \in \Sigma_{\theta}$, $\w := (z^{\alpha}I + A  )^{-1}f$, $\w_h := (z^{\alpha}I + A_h  )^{-1}P_h f$. Then there exists a positive constant $C(s,n,\theta)$ such that
$$
\|\w - \w_h \|_{L^2(\W)} + h^{\gamma}|\w - \w_h |_{H^s(\rn)} \leq   Ch^{2 \gamma }\|f\|_{L^2(\W)}.
$$
As before, $\gamma = \min \{ s, 1/2 - \eps\}$, with $\varepsilon > 0$ arbitrary small. 
\end{lemma}  
\begin{proof} 
We consider first the case $s \geq 1/2$. By definition of $\w$ and $\w_h$, it holds that
$$
\begin{aligned}
z^{\alpha}(\w , \varphi) + \langle \w , \varphi \rangle_s  & = (f , \varphi ),  \quad \forall \varphi \in \widetilde{H}^s(\W),\\
z^{\alpha}(\w_h , \varphi) + \langle \w_h , \varphi \rangle_s  & = (f , \varphi ),  \quad \forall \varphi \in X_h.
\end{aligned}
$$

If we set $e_h := \w - \w_h$ and subtract these two expressions, we derive
\begin{equation}
\label{lema_eq_0}
z^{\alpha}(e_h , \varphi) + \langle e_h , \varphi \rangle_s  = 0,  \quad \forall \varphi \in X_h.
\end{equation}

Applying  Lemma \ref{lemaTecnico} and this identity, we arrive to
$$ \begin{aligned}
\left| z^{\alpha} \right| \|e_h\|^2_{L^2(\W)} + |e_h|^2_{H^s(\rn)} & \leq  c\left| z^{\alpha} (e_h,e_h) + \langle e_h , e_h \rangle_s \right| \\
 & = c\left| z^{\alpha} (e_h,\w -  \varphi) + \langle e_h , \w -  \varphi\rangle_s \right| \quad \forall \varphi \in X_h .
\end{aligned}$$

Taking $\varphi = \Pi_h \w$ in the former expression, where $\Pi_h$ is a suitable quasi-interpolation operator (see, for example, \cite[Section 4.1]{AcostaBorthagaray}), we deduce
\begin{equation}
\label{lema_eq1} \begin{split}
\left| z^{\alpha} \right| & \|e_h\|^2_{L^2(\W)} + |e_h|^2_{H^s(\rn)}  \\
& \leq c \left(   \|e_h\|_{L^2(\W)}h^{1/2 - \varepsilon}|\w|_{H^s(\rn)} + |e_h|_{H^s(\rn)} h^{1/2 - \varepsilon}|\w|_{H^{s+1/2-\eps}(\rn)}  \right),
\end{split}
\end{equation}
where we have used the fact that $s \geq 1/2$ implies  that $h^s \leq h^{1/2 - \varepsilon}$. 

On the other hand, if we choose $\varphi = \w$ in Lemma \ref{lemaTecnico}, we obtain
$$
\begin{aligned}
\left| z^{\alpha} \right| \|\w\|^2_{L^2(\W)} +| \w |^2_{H^s(\rn)} & \leq c\left| z^{\alpha}  (\w , \w) + \langle \w , \w \rangle_s  \right| \\
& = c \left| (f,\w)  \right| \leq c \|f\|_{L^2(\W)}\|\w\|_{L^2(\W)}.\end{aligned}
 $$ 
Consequently, 
\begin{equation}
\label{lema_eq2}
 \|\w\|_{L^2(\W)} \leq c \left| z \right|^{-\alpha} \|f\|_{L^2(\W)}\quad \text{and} \quad  |\w|_{H^s(\rn)} \leq c \left| z \right|^{-\alpha/2} \|f\|_{L^2(\W)}.
\end{equation}

From Proposition \ref{prop:regHr}, we know that for the case $z = 0$ the estimate $|\w|_{H^{s+1/2-\eps}(\rn)} \leq \| f \|_{L^2(\W)}$ holds. Utilizing this estimate with $-z^{\alpha}\w + f$ instead of $f$, we obtain 
$$
|\w|_{H^{s+1/2-\eps}(\rn)} \leq \|-z^{\alpha}\w +  f \|_{L^2(\W)} \leq c\| f \|_{L^2(\W)},
$$
where in the last inequality we used \eqref{lema_eq2}. Combining this with \eqref{lema_eq1}, we derive
\begin{equation*}
\left| z^{\alpha} \right| \|e_h\|^2_{L^2(\W)} + |e_h|^2_{H^s(\rn)} \leq c h^{1/2 - \varepsilon}\| f \|_{L^2(\W)}  \left(   z^{\alpha/2}\|e_h\|_{L^2(\W)} + |e_h|_{H^s(\rn)}  \right).
\end{equation*}
This implies that 
\begin{equation}
\label{lema_eq3}
\left| z^{\alpha} \right| \|e_h\|^2_{L^2(\W)} + |e_h|^2_{H^s(\rn)} \leq c h^{1 - 2\varepsilon}\| f \|^2_{L^2(\W)},
\end{equation}
and gives the bound for $|e_h|_{H^s(\rn)}$. Next, we aim to  estimate $\|e_h\|^2_{L^2(\W)}$. For this purpose, we proceed via the following duality argument. Given $\varphi \in L^2(\W)$, define 
$$
\psi := (z^{\alpha} + A)^{-1}\varphi \text{  and  } \psi_h := (z^{\alpha} + A)^{-1}P_h\varphi. 
$$
Thus, we write
$$
\|e_h\|_{L^2(\W)} = \sup_{ \varphi \in L^2(\W) } \frac{\left|(e_h,\varphi)\right|}{\|\varphi\|_{L^2(\W)}} = \sup_{ \varphi \in L^2(\W) } \frac{\left|z^{\alpha}(e_h,\psi) + \langle e_h,\psi \rangle_s \right|}{\|\varphi\|_{L^2(\W)}}.
$$ 
We aim to bound the supremum in the identity above. Resorting to \eqref{lema_eq_0} and the Cauchy-Schwarz inequality, we bound
$$ 
\begin{aligned}
 \Big|z^{\alpha} & (e_h, \psi) + \langle e_h,\psi \rangle_s \Big|  = \left|z^{\alpha}(e_h,\psi - \psi_h) + \langle e_h,\psi - \psi_h \rangle_s \right| \\
& \leq z^{\alpha/2}\|e_h\|_{L^2(\W)} z^{\alpha/2}\|\psi - \psi_h\|_{L^2(\W)} + |e_h|_{H^s(\rn)} |\psi - \psi_h|_{H^s(\rn)} \\
& \leq \left( z^{\alpha/2}\|e_h\|_{L^2(\W)} + |e_h|_{H^s(\rn)} \right) \left( z^{\alpha/2}\|\psi - \psi_h\|_{L^2(\W)} +|\psi - \psi_h|_{H^s(\rn)} \right) .
\end{aligned} 
$$
Finally, applying \eqref{lema_eq3} we arrive at 
\[
\left|z^{\alpha}(e_h,\psi) + \langle e_h,\psi \rangle_s \right| \le h^{1 - 2\varepsilon}\|f\|_{L^2(\W)}\|\varphi\|_{L^2(\W)},
\]
from where we can derive the desired inequality.

The analysis of the case $s \leq 1/2$ can be carried out in analogously. Indeed, using 
that $h^s \geq h^{1/2 - \eps}$ we obtain, instead of \eqref{lema_eq1}, the inequality
\begin{equation*} \begin{aligned}
\left| z^{\alpha} \right| \|e_h\|^2_{L^2(\W)} + & |e_h|^2_{H^s(\rn)} \\
&  \leq c \left(   \|e_h\|_{L^2(\W)}h^{s }|\w|_{H^s(\rn)} + |e_h|_{H^s(\rn)} h^{s }|\w|_{H^{s+1/2-\eps}(\rn)}  \right),
\end{aligned} \end{equation*}
and proceeding as before we arrive at the desired  estimate. 
\end{proof}

At this point, we are able to give a sketch of the proof of Theorem \ref{teo:semi}.

\begin{proof}[Proof of Theorem \ref{teo:semi}] The proof outlined in \cite[Theorem 3.2]{Jin} can be reproduced using Lemma \ref{lemaError} instead of Lemma 3.1 from that work and approximation properties of the elliptic projection along with the estimate 
$
\| (-\Delta)^s w\|_{L^2(\W)} \le C | w |_{H^{2s}(\rn)}.
$
The details are therefore omitted. 
\end{proof}

%


\jp{Finally, we consider the problem with zero initial conditions and a non-zero source term.}

\begin{proof}[Proof of Theorem \ref{teo:semi_nonhom}]  Using the approximation properties of the elliptic projection together with an appropriate generalization of \cite[Lemma 2.3]{errorF} (in the same spirit of Lemma \ref{lemaError}), the proof given in that work can be adapted to our purpose without major difficulties. 
\end{proof}

\bibliography{bib_parabolico}{}
\bibliographystyle{plain}
\end{document}